\documentclass[10pt,a4paper]{amsart}%
\usepackage{amsfonts,amsmath,amssymb}
\usepackage{hyperref}
\usepackage{amssymb,amsthm,amsxtra}
\usepackage[usenames]{color}
\usepackage{amscd}
\usepackage{amsthm}
\usepackage{amsfonts}
\usepackage{amssymb}
\usepackage{amsmath}
\usepackage{graphicx}
\usepackage{enumerate}%
\providecommand{\U}[1]{\protect\rule{.1in}{.1in}}
\newcommand{\fg}{\mathfrak{g}}
\newcommand{\fv}{\mathfrak{v}}
\newtheorem*{theorem*}{Theorem}
\newtheorem*{lemma*}{Lemma}
\newtheorem{lemma}{Lemma}[subsection]
\newtheorem{proposition}[lemma]{Proposition}
\newtheorem{remark}[lemma]{Remark}
\newtheorem{example}[lemma]{Example}
\newtheorem{theorem}[lemma]{Theorem}

\newtheorem{corollary}[lemma]{Corollary}

\newtheorem*{conjecture*}{Conjecture}
\newtheorem*{remark*}{Remark}
\newtheorem{thm}[lemma]{Theorem}
\newtheorem{prop}[lemma]{Proposition}
\newtheorem{lem}[lemma]{Lemma}

\newtheorem{cor}[lemma]{Corollary}

\newtheorem{introtheorem}{Theorem}

\newtheorem{introthm}[introtheorem]{Theorem}

\baselineskip 18pt \textwidth 16cm  \theoremstyle{plain}

\newcommand{\eps}{\varepsilon}

\newcommand{\Z}{{\mathbb Z}}

\newcommand{\C}{{\mathbb C}}

\newcommand{\g}{{\mathfrak{g}}}

\newcommand{\fk}{{\mathfrak{k}}}
\newcommand{\fa}{{\mathfrak{a}}}

\newcommand{\cO}{{\mathcal{O}}}

\newcommand{\onto}{\twoheadrightarrow}

\newcommand{\cN}{\mathcal{N}}

\newcommand{\cV}{\mathcal{V}}

\newcommand{\fl}{\mathfrak{l}}

\begin{document}

\author{Dmitry Gourevitch}
\address{Dmitry Gourevitch, Faculty of Mathematics and Computer Science, Weizmann
Institute of Science, POB 26, Rehovot 76100, Israel }
\email{dimagur@weizmann.ac.il}
\urladdr{\url{http://www.wisdom.weizmann.ac.il/~dimagur}}
\author{Siddhartha Sahi}
\address{Siddhartha Sahi, Department of Mathematics, Rutgers University, Hill Center -
Busch Campus, 110 Frelinghuysen Road Piscataway, NJ 08854-8019, USA}
\email{sahi@math.rugers.edu}
\date{\today}
\title{Degenerate Whittaker functionals for real reductive groups}

\begin{abstract}
In this paper we establish a connection between the associated variety of a
representation and the existence of certain \emph{degenerate} Whittaker
functionals, for both smooth and K-finite vectors, for all quasi-split real
reductive groups, thereby generalizing results of Kostant, Matumoto and others.

\end{abstract}
\maketitle

\section{Introduction}

%Let $\GR$ be a real reductive group, $K$ be the maximal compact subgroup of $\GR$, and $P=MAN$ be a minimal parabolic subgroup of $\GR$.
%%$N$ be the nilradical of $P$ and .
%Let ${\mathfrak{g}}, \mathfrak{p}, \fk,\fa$ and $\mathfrak{n}$ be the complexified Lie
%algebras of $\GR,P,K,A$ and $N$. Then $\fn$ will be the span of root spaces for a system of positive roots $\Delta^+$. Denote by $\Sigma \subset \Delta^+$ the set of simple roots.
%Let $\Psi$ denote the space of characters of $\fn$. Then $\Psi =[\n,\n]^{\bot} \subset \fn^*, \, \Psi$ is spanned by the root spaces for $\Sigma$ and $\C[\Psi]=S(\n/[\n,\n]).$

Let $G$ be a real reductive group with Cartan involution $\theta$ and maximal
compact subgroup $K=G^{\theta}$. We denote the Lie algebras of $G,K$ by
$\mathfrak{g}_{0},{\mathfrak{k}}_{0}$ and their complexifications by
$\mathfrak{g},{\mathfrak{k,}}$ and analogous notation will be applied without
comment to Lie algebras of other groups below. Let $\mathcal{M=M}(G)$ be the
category of smooth admissible {Fr\'{e}chet }$G$-representations of moderate
growth, and let $\mathcal{HC}=\mathcal{HC}\left(  G\right)  =\mathcal{HC}%
(\mathfrak{g},K)$ be the category of Harish-Chandra modules (finitely
generated admissible $(\mathfrak{g},K)$-modules). We will denote a typical
representation in $\mathcal{M}(G)$ by $\left(  \pi,W\right)  $ (or $\pi$ or
$W$) and a representation in $\mathcal{HC}\left(  G\right)  $ by $\left(
\sigma,M\right)  $ (or $\sigma$ or $M$). By \cite[Chapter 11]{Wal2} or
\cite{CasGlob} or \cite{BerKr} we have an equivalence of categories
\[
\left(  \pi,W\right)  \mapsto\left(  \pi^{HC},W^{HC}\right)  :\mathcal{M}%
\rightarrow\mathcal{HC}%
\]
where $\left(  \pi^{HC},W^{HC}\right)  $ denotes the Harish-Chandra module of
$K$-finite vectors in $\left(  \pi,W\right)  .$

We assume throughout this paper that $G$ is \textbf{quasisplit}. We fix a
Borel subgroup $B$ with nilradical $N$ and $\theta$-stable maximally split
Cartan subgroup $H=TA$, and we define%
\begin{equation}
\mathfrak{n}^{\prime}=\left[  \mathfrak{n},\mathfrak{n}\right]
,\mathfrak{v=n/n}^{\prime},\Psi=\mathfrak{v}^{\ast}\subset\mathfrak{n}^{\ast
},\Psi_{0}=\left\{  \psi\in\Psi:\psi\left(  x\right)  \in i\mathbb{R}\text{
for }x\in\mathfrak{n}_{0}\right\}  . \label{=v}%
\end{equation}
Thus $\Psi$ is the space of Lie algebra characters of $\mathfrak{n}$ or
equivalently, via the exponential map, group characters of $N$, while
$\Psi_{0}$ corresponds to unitary characters of $N$.
%?? Dima
Note that $\fv$ is the direct sum of simple root spaces, and thus $H_{\mathbb{C}}$ has finitely many orbits on $\fv$ and on
 $\Psi=\fv^*$ (see \S \ref{subsec:ResRoots} for more details).
%?? Dima
We say that $\psi$ is \emph{non-degenerate }if its $H_{\mathbb{C}}$-orbit is
open in $\Psi$. We define $\Psi^{\times}$ to be the set of non-degenerate
characters, and set $\Psi_{0}^{\times}=\Psi^{\times}\cap\Psi_{0}.$

For $\psi\in\Psi,\, \pi\in\mathcal{M}(G)$ and $\sigma\in\mathcal{HC}(G)$ we
define the corresponding Whittaker spaces as follows%
\begin{align}
Wh_{\psi}^{\ast}(\pi)  &  :=\operatorname{Hom}_{N}^{ct}(\pi,\psi),\Psi
(\pi):=\{\psi\in\Psi:Wh_{\psi}^{\ast}(\pi)\neq0\},\\
Wh_{\psi}^{\prime}\left(  \sigma\right)   &  :=\operatorname{Hom}%
_{\mathfrak{n}}(\sigma,\psi),\Psi(\sigma):=\{\psi\in\Psi:Wh_{\psi}^{\prime
}\left(  \sigma\right)  \neq0\},
\end{align}
where $\operatorname{Hom}_{N}^{ct}\left(  \cdot\right)  $ denotes the space of
\emph{continuous} $N$-homomorphisms (functionals). We also define%
\begin{align*}
\Psi^{\times}\left(  \pi\right)   &  =\Psi(\pi)\cap\Psi^{\times},\Psi
_{0}\left(  \pi\right)  =\Psi(\pi)\cap\Psi_{0},\Psi_{0}^{\times}\left(
\pi\right)  =\Psi(\pi)\cap\Psi_{0}^{\times}\\
\Psi^{\times}\left(  \sigma\right)   &  =\Psi(\sigma)\cap\Psi^{\times}%
,\Psi_{0}\left(  \pi\right)  =\Psi(\sigma)\cap\Psi_{0},\Psi_{0}^{\times
}\left(  \sigma\right)  =\Psi(\sigma)\cap\Psi_{0}^{\times}%
\end{align*}
If $\left(  \pi,W\right)  \in\mathcal{M}(G)$ then $W^{HC}$ is dense in $W$ and
thus
\[
Wh_{\psi}^{\ast}(\pi)\subset Wh_{\psi}^{\prime}\left(  \pi^{HC}\right)  \text{
and }\Psi(\pi)\subset\Psi(\pi^{HC}) .
\]

We say that $\pi$ (resp.  $\sigma$) is \emph{generic }if $\Psi^{\times}\left(
\pi\right)  $ (resp. $\Psi^{\times}\left(
\sigma\right)  $) is not empty. By \cite[Theorem 8.2]{CHM} we have $\Psi
^{\times}\left(  \pi\right)  =\Psi_{0}^{\times}\left(  \pi\right)  $. In fact
using the same argument one can show $\Psi\left(  \pi\right)  =\Psi_{0}\left(
\pi\right)  $, but we will not use this result.

Let $\mathcal{N}\subset\mathfrak{g}^{\ast}$ denote the nilpotent cone, and
define%
\[
\mathcal{N}_{\theta}=\mathcal{N}\cap\mathfrak{k}^{\perp},\mathcal{N}%
_{0}=\mathcal{N}\cap\mathfrak{g}_{0}^{\ast}\text{.}%
\]
To a representation $\pi$ or $\sigma$ one can attach invariants such as the
annihilator variety, associated variety and wavefront set (see
\S \ref{subsec:ASVars} below)
\[
\mathrm{An}\mathcal{V}\left(  \cdot\right)  \subset\mathcal{N}\text{,
}\mathrm{As}\mathcal{V}\left(  \cdot\right)  \subset\mathcal{N}_{\theta
}\text{, }\mathrm{WF}\left(  \cdot\right)  \subset i\mathcal{N}_{0}%
\]
The dimension of these invariants determines the size (Gelfand-Kirillov
dimension) of the representation. We say that $\pi$ or $\sigma$ is
\emph{large} if $\mathrm{its}$ annihilator variety is all of $\mathcal{N}$. A
key result of Kostant \cite{Kos} proves that a representation is large if and
only if it is generic. More precisely for $\pi\in\mathcal{M}(G)$ one has
\begin{align*}
\mathrm{An}\mathcal{V}\left(  \pi\right)   &  =\mathrm{An}\mathcal{V}\left(
\pi^{HC}\right)  =\mathcal{N}\iff\mathrm{As}\mathcal{V}\left(  \pi
^{HC}\right) \text{ is open in }  \mathcal{N}_{\theta}\iff\mathrm{WF}\left(  \pi\right)
\text{ is open in }i\mathcal{N}_{0}\\
&  \iff\Psi_{0}^{\times}\left(  \pi\right)  \neq\emptyset\iff\Psi^{\times
}\left(  \pi\right)  \neq\emptyset\iff\Psi^{\times}\left(  \pi^{HC}\right)
\neq\emptyset
\end{align*}

A number of papers (e.g. \cite{GW,Mat,MatDuke,MatActa}) provide certain
generalizations of \cite{Kos} to non-generic representations; namely, they
consider functionals equivariant with respect to non-degenerate characters of
nilradicals of other parabolic subgroups, often referred to as
\emph{generalized} Whittaker functionals. In this paper we study a different
type of analog: we consider functionals equivariant with respect to possibly
degenerate characters of the nilradical of the standard Borel subgroup.
Following Zelevinsky \cite[\S 8.3]{Zl} we refer to these as
\emph{degenerate }Whittaker functionals.

\subsection{Main results}

%Let $M$ be an irreducible Harish-Chandra
%$({\mathfrak{g}},K)$-module. For reductive groups different from $GL_{n}$ we
%do not have enough Levi subgroups to repeat the proof in section
%\ref{sec:PfSuppFunc2}. In particular, $\NV(\pi)$ can be a closure of a
%non-Richardson nilpotent orbit.
%Matumoto's result (Theorem \ref{thm:Mat}) gives an upper bound on the set of characters of $\fn$ that define a non-zero Whittaker space on a given irreducible Harish-Chandra module. %
%However, in this section we compute the variety $\Psi_{M}$ of
%characters of $\mathfrak{n}$ that define a non-zero (degenerate) Whittaker
%space for $M$, based on Proposition \ref{prop:Osborne} and on Nakayama's lemma.

\begin{introtheorem}
\label{thm:Main} Let $pr_{\mathfrak{n}^{\ast}}:\mathfrak{g}^{\ast}%
\rightarrow\mathfrak{n}^{\ast}$ denote the restriction to $\mathfrak{n}$, then
for $\sigma\in\mathcal{HC}$ we have%
\[
\Psi(\sigma)=pr_{\mathfrak{n}^{\ast}}(\mathrm{As}\mathcal{V}\left(
\sigma\right)  )\cap\Psi\text{.}%
\]

\end{introtheorem}

This is proved in section \ref{sec:PfMain} below. We now describe the
connection between $\Psi_{0}(\pi)$, $\Psi_{0}(\pi^{HC})$ and the wavefront set
$\mathrm{WF}(\pi)$. Let $H=TA$ be the maximally split Cartan subgroup of $G$
as above, and define
\begin{equation}
F=F_{G}:=\left\{  ad(x) \mid x \in\exp i{\mathfrak{a}}_{0} \text{ and } ad\left(  x\right)
^{2}=1\right\} \subset Int(\g_{\C}) \label{=FG}%
\end{equation}
It is easy to see that $F_{G}$ is a finite group of order $2^{r_0}$, where $r_0$ is the real rank of $G$ (see Lemma \ref{lem:FGIso}). Moreover, it commutes with the Cartan involution and complex conjugation and therefore preserves $\g_0$ and $\fk$ (see  \cite[\S I.1]{KR}).

\begin{introtheorem}
\label{thm:OurMat} Let $\pi\in\mathcal{M}$ and write $\sigma=\pi^{HC}$; then
we have%
\begin{equation}
\Psi_{0}(\pi)\subset\operatorname{WF}(\pi)\cap\Psi\subset F_G \cdot \Psi_{0}\left(
\pi\right)  =\Psi_{0}(\sigma)\text{.} \label{=ourmat1}%
\end{equation}
Moreover if $G=GL_{n}\left(  {\mathbb{R}}\right)  $ or if $G$ is a complex
group then %$\widetilde{\pi}=\pi$ and
we have
\begin{equation}
\Psi_{0}(\pi)=\operatorname{WF}(\pi)\cap\Psi=\Psi_{0}(\sigma)=\mathrm{An}%
\mathcal{V}(\sigma)\cap\Psi_{0}. \label{=ourmat2}%
\end{equation}

\end{introtheorem}

If $\pi$ is generic, then Theorem \ref{thm:OurMat} follows immediately from
Theorem \ref{thm:Main} and \cite[Theorem A]{MatComp}. We prove the general
result by reduction to the generic case using the Kostant-Sekiguchi
correspondence, the coinvariants functor $C_{\mathfrak{u}}$, where
$\mathfrak{u}$ is the nilradical of a suitable parabolic subalgebra, \cite{SV} and Theorem \ref{thm:Main}.

Theorem \ref{thm:OurMat} implies $\mathrm{An}\mathcal{V}(\sigma)\cap\Psi
_{0}\supset\Psi_{0}(\sigma)$ though the reverse inclusion can fail, as shown
in section \ref{sec:RemOurMat} for the group $U\left(  2,2\right)  $. We
conjecture however that for all quasi-split groups one has the equality
\begin{equation}
\Psi_{0}(\pi)=\mathrm{WF}(\pi)\cap\Psi, \label{=intersect}%
\end{equation}
although the proof probably requires additional arguments of an analytic nature.

We prove a stronger result if $G=GL_{n}\left(  {\mathbb{R}}\right)  $ or if
$G$ is a complex classical group, \emph{i.e.} one of the groups%
\begin{equation}
GL_{n}\left(  {\mathbb{C}}\right)  ,SL_{n}\left(  {\mathbb{C}}\right)
,O_{n}\left(  {\mathbb{C}}\right)  ,SO_{n}\left(  {\mathbb{C}}\right)
,Sp_{n}\left(  {\mathbb{C}}\right)  \text{.} \label{=ccg}%
\end{equation}

\begin{introthm}
\label{thm:determine}Let $\pi\in\mathcal{M}\left(  G\right)  $ and suppose one
of the following holds:

\begin{enumerate}[(a)]
\item \label{it:detGL}$G=GL_{n}\left(  {\mathbb{R}}\right)  ,GL_{n}\left(  {\mathbb{C}%
}\right)  $ or $SL_{n}\left(  {\mathbb{C}}\right)  $;

\item \label{it:detC} $G=O_{n}\left(  {\mathbb{C}}\right)  ,SO_{n}\left(  {\mathbb{C}}\right)
$, or $Sp_{n}\left(  {\mathbb{C}}\right)  $ and $\pi$ is \emph{irreducible;}
\end{enumerate}

then $\Psi_{0}\left(  \pi\right)  $ and $\mathrm{WF}\left(  \pi\right)  $
determine each other uniquely.
\end{introthm}

Part \eqref{it:detGL} of Theorem \ref{thm:determine} follows easily from Theorem
\ref{thm:OurMat}, since for the groups in this case, every nilpotent orbit
intersects $\Psi_{0}$. This enables us to strengthen several results from
\cite{AGS}. We note that for \emph{unitarizable} $\pi$, a weaker version of
this theorem follows from \cite[Theorem A]{GS}.

For the groups in part \eqref{it:detC} of Theorem \ref{thm:determine} not every
nilpotent orbit intersects $\Psi_{0}$, however if $\pi$ is irreducible then
$\mathrm{An}\mathcal{V}\left(  \pi\right)  $ is the closure of a single
nilpotent orbit (see \cite{JosVar}), and this allows us to deduce  part \eqref{it:detC}  from the following result that may be of independent interest.

\begin{introthm}
\label{thm:ClassGroup} Every nilpotent orbit ${\mathcal{O}}$ for a complex
classical group is uniquely determined by $\overline{{\mathcal{O}}}\cap\Psi$.
\end{introthm}

If $\pi$ is not irreducible then $\mathrm{WF}(\pi)$ might be the union of
several orbit closures, and as shown in (\ref{CounterClass}) such a union is
not determined by its intersection with $\Psi_{0}$. We also note that Theorem
\ref{thm:ClassGroup} does not hold for \emph{any} exceptional Lie group and we describe
all the counterexamples in section \ref{subsec:ExcGroups}.
%?? Dima
This shows that Theorem \ref{thm:determine} cannot be strengthened since if
$G$ is a complex semi-simple group then any coadjoint nilpotent orbit in
${\mathfrak{k}}^{\bot}$ is the associated variety of a Harish-Chandra module
(see \cite[Theorem 10.3.4]{CM}).

Let $P=LU\subset G$ be a standard parabolic subgroup of $G$ and $\pi_{P}$
denote the Jacquet restriction of $\pi$ to $P$. Then it is easy to see that
$\Psi(\pi_{P})=\Psi(\pi)\cap\mathfrak{l}^{*}$. In \S \ref{subsec:WFJac} we use
this fact and Theorems \ref{thm:OurMat} and \ref{thm:determine} to show that
under certain conditions
\begin{equation}\label{=JacRes}
WF(\pi_{P})=WF(\pi)\cap\mathfrak{l}^{*}.
\end{equation}
It would be interesting  to know whether  this equality holds in general.
%?? Dima

\begin{remark*}
Over p-adic fields, the associated and annihilator varieties are not defined
but the notion of wave front set still makes sense (see \cite{Hch,HowGL,Rod}).
In \cite{MW}, the authors give a very general definition of degenerate
Whittaker spaces and prove that the dimensions of \textquotedblleft minimally
degenerate\textquotedblright\ Whittaker spaces equal the multiplicities of
corresponding coadjoint nilpotent orbits in the wave front set.
%?? Dima
The technique of \cite{MW} relies on approximation of unipotent subgroups by
open compact subgroups and thus is not applicable in the archimedean case.
%?? Dima

\end{remark*}

\subsection{Structure of the paper}

In section \ref{sec:Prel} we give several necessary definitions and
preliminary results on
%geometry of coadjoint orbits,
filtrations, associated/annihilator varieties, Whittaker functionals, and
discuss a version of the Casselman-Jacquet functor.

In section \ref{sec:PfMain} we prove Theorem \ref{thm:Main}. Let $\left(
\sigma,M\right)  $ be a Harish-Chandra module for $G,$ then every good
${\mathfrak{g}}$-filtration on $M$ is good as an $\mathfrak{n}$-filtration.
This implies that $\mathrm{As}\mathcal{V}_{\mathfrak{n}}(M)=pr_{\mathfrak{n}%
^{\ast}}\left[  \mathrm{As}\mathcal{V}_{\mathfrak{g}}(M)\right]  $ where
$\mathrm{As}\mathcal{V}_{\mathfrak{n}}(M)$ denotes the associated variety of
$M$ as an $\mathfrak{n}$-module (by restriction). We next pass to the
commutative Lie algebra $\mathfrak{v=n}/\mathfrak{n}^{\prime}$ by considering
the module of coinvariants
\[
CM=C\left(  M\right)  =C_{\mathfrak{n}^{\prime}}\left(  M\right)
:=M/\mathfrak{n}^{\prime}M.
\]
Since $\mathfrak{v}$ is commutative, $\mathrm{As}\mathcal{V}_{\mathfrak{v}%
}\left(  CM\right)  =\mathrm{An}\mathcal{V}_{\mathfrak{v}}\left(  CM\right)  $
and we denote both by $\mathcal{V}_{\mathfrak{v}}\left(  CM\right)  $. Then as
shown in Lemma \ref{lem:AVVSupp},$\mathcal{V}_{\mathfrak{v}}\left(  CM\right)
=\operatorname*{Supp}\left(  CM\right)  $, which further coincides with
$\Psi(M)$ by the Nakayama Lemma (see \S \ref{subsec:Nakayama}).

If $M$ is any finitely generated $\mathfrak{n}$-module, $\mathcal{V}%
_{\mathfrak{v}}\left(  CM\right)  \subset\mathrm{As}\mathcal{V}_{\mathfrak{n}%
}(M)\cap\Psi$, and our task is to prove that%
\[
\mathcal{V}_{\mathfrak{v}}\left(  CM\right)  \supset\mathrm{As}\mathcal{V}%
_{\mathfrak{n}}(M)\cap\Psi.
\]
This is \emph{not} true for a general finitely-generated $\mathfrak{n}$-module
$V$, indeed $C(V)$ could even vanish (see \S \ref{subsubsec:CEx}). However, if
$M$ is a Harish-Chandra module it was proven by Casselman that even
$M/\mathfrak{n}M$ is non-zero, indeed he proved that $\cap\mathfrak{n}^{i}%
M=0$. This implies that $M$ imbeds (densely) into its $\mathfrak{n}$-adic
completion $\widehat{M}:=\lim\limits_{\leftarrow}M/\mathfrak{n}^{i}M$.
Following \cite{ENV} we let%
\[
JM=J\left(  M\right)  :=\widehat{M}^{\mathfrak{h}\text{-finite}}%
\]
denote the submodule of $\mathfrak{h}$-finite vectors. The functor $J$ can be
applied to both $M$ and $C(M)$ and we prove that%
\[
\mathcal{V}_{\mathfrak{v}}\left(  CM\right)  =\mathcal{V}_{\mathfrak{v}%
}\left(  J\left(  CM\right)  \right)  =\mathcal{V}_{\mathfrak{v}}\left(
C\left(  JM\right)  \right)  .
\]
The first equality follows from the fact that $CM$ and $J\left(  CM\right)  $
are both dense in $\left(  CM\right)  _{\left[  \mathfrak{n}\right]  }$ and
hence have the same annihilator, while the second follows from the isomorphism
$J\left(  CM\right)  \simeq C\left(  JM\right)  $ proved in Lemma
\ref{lem:QuotJac}. Moreover $JM$ is finitely generated over $\mathfrak{n}$ and
glued from lowest weight modules, and hence we get (by Lemma \ref{lem:JacQuot}%
)
\[
\mathcal{V}_{\mathfrak{v}}(C\left(  JM\right)  )=\mathcal{V}_{\mathfrak{v}%
}(JM)\cap\Psi
\]

This reduces the problem to showing
\[
\mathcal{V}_{\mathfrak{v}}(JM)\cap\Psi\supset\mathrm{As}\mathcal{V}%
_{\mathfrak{n}}(M)\cap\Psi,
\]
which we prove in section \ref{subsec:PfAVJ}, using the main result of
\cite{ENV} that describes $J(M)$ as a deformation of $M$. The description is
in the language of $D$-modules, using the Beilinson-Bernstein localization.
%?? Dima
In this language, the above-mentioned deformation is a certain near-by cycle.
%?? Dima
While it is not true in general that the operation of taking associated
variety commutes with limits, but this was proven to be true for holonomic
$D$-modules with regular singularities in \cite{Gin}. This implies the above
containment and finishes the proof of Theorem \ref{thm:Main}.

In section \ref{sec:PfOurMat} we first prove Theorem \ref{thm:OurMat}. The
special case of large representations follows from \cite{MatActa,MatComp}. To
reduce to this case, we note in \S \ref{subsec:JacRes} that any unitary character $\psi$ of $N$ defines a
parabolic subgroup $P=LU$ such that $\psi$ is trivial on $N\cap U$ and
non-degenerate on $N\cap L$. Thus we consider the $U$-coinvariants of $\pi$,
and we need to know when this space is large as a representation of the Levi
subgroup $L$. For that purpose we use Theorem \ref{thm:Main}. We also use
\cite{SV} that shows that the wave-front set corresponds to the associated
variety via the Kostant-Sekiguchi bijection. We next use Theorem
\ref{thm:OurMat} to reduce the proof of Theorem \ref{thm:determine} to Theorem
\ref{thm:ClassGroup}. In subsection \ref{subsec:WFJac} we deduce from Theorem
\ref{thm:OurMat} the formula \eqref{=JacRes} for the wave front set of Jacquet restriction.

In section \ref{sec:GL} we give several consequences of Theorem
\ref{thm:determine}\eqref{it:detGL}, including applications to the theory of derivatives of
representations of $GL(n)$.
%?? Dima
More precisely, we give a formula for the annihilator variety of the
derivative $B^{k}(\pi),$ defined in \cite{AGS}, in terms of the annihilator
variety of $\pi$. Over non-archimedean fields, where derivatives were
originally defined by Bernstein and Zelevinsky, an analogous formula is not
possible since cuspidal representations have full wave-front set, while all
their derivatives except the last one vanish.
%?? DIma

In section \ref{sec:Orb} we prove Theorem \ref{thm:ClassGroup}, using basic
results on nilpotent orbits from \cite{CM,Car}.

\subsection{Acknowledgements}

We are grateful to Avraham Aizenbud, Dan Barbasch, Joseph Bernstein, Victor
Ginzburg, Anthony Joseph, Maxim Leyenson, Kari Vilonen, and David Vogan for
fruitful discussions. We thank the referee for useful remarks. Part of the work on this paper was done during the
program \textquotedblleft Analysis on Lie Groups" at the Max Planck Institute
for Mathematics (MPIM) in Bonn, and we thank the organizers of the program -
Bernhard Kroetz, Eitan Sayag and Henrik Schlichtkrull,
%?? Check spelling : Schlichtkrull
and the administration of the MPIM for their hospitality.
The second-named author also thanks the Weizmann Institute for its hospitality.

D.G. was partially supported by  ERC grant 291612,  ISF grant 756/12 and a Minerva foundation grant.

\section{Preliminaries}

\label{sec:Prel}

\setcounter{lemma}{0}

\subsection{The Nakayama lemma}

\label{subsec:Nakayama}

The classical Nakayama lemma is a commutative analog of the problems
considered in this paper, and in this section we explain this point of view.

Let $A$ be a commutative algebra, finitely generated over $\C$. The
characters of $A$ are ring homomorphisms $A\rightarrow{\mathbb{C}}$. Such a
homomorphism is uniquely described by its kernel, which is a maximal ideal in
$A$. Conversely, by Hilbert's Nullstellensatz, every maximal ideal
$\mathfrak{m}\in\operatorname*{Max}A$ is the kernel of a (unique) ring
homomorphism $\phi_{\mathfrak{m}}:A\rightarrow{\mathbb{C}}$. For an $A$-module
$M$ and $\mathfrak{m}\in\operatorname{Max}A$, we have $\operatorname{Hom}%
_{A}(M,\phi_{\mathfrak{m}})\cong M/\mathfrak{m}M$. Thus, we can define
\[
\Psi(M):=\{\mathfrak{m}\in\operatorname{Max}A\mid M\neq\mathfrak{m}M\}.
\]
On the other hand, the support of $M$ is defined to be
\[
\mathrm{Supp}(M):=\operatorname{Var}(\operatorname{Ann}M),
\]
where $\operatorname{Ann}M=\{a\in A\mid aM=0\}$ is the annihilator ideal of $M$ and $\operatorname{Var}$ denotes the variety of zeroes.

\begin{lem}
[Nakayama]\label{lem:Nak} If $M$ is finitely generated over $A$ then
$\mathrm{Supp}(M)=\Psi(M)$.
\end{lem}

For completeness, we deduce this result from the version in \cite{AM}.

\begin{proof}

The lemma follows from the following chain of equivalences which hold  for any $\mathfrak{m}\in\operatorname{Max}A$:
\begin{eqnarray*}
\mathfrak{m}M=M & \Leftrightarrow&  \exists x \in \mathfrak{m}
\text{ that acts by }1 \text{ on } M \quad \text{\cite[Corollary 2.5]{AM}} \\
&\Leftrightarrow&
1\in\mathfrak{m}+\operatorname{Ann}M \Leftrightarrow
\operatorname{Ann}M\nsubseteq\mathfrak{m}
\end{eqnarray*}
\end{proof}

%?? Dima
Let us now describe what we consider a commutative analog of Theorem
\ref{thm:Main}. Let $B\subset A$ be a subalgebra. Then we have a natural map
$\phi:\operatorname{Spec} A \twoheadrightarrow\operatorname{Spec} B$, and the
support of $M$ considered as a $B$-module is the image of the support of $M$
in $\operatorname{Spec} A$. Further, if $I\subset B$ is a (radical) ideal then
we have a natural embedding $\operatorname{Spec}(B/I) \subset
\operatorname{Spec}(B)$ and $\Psi_B(M/IM)=\Psi_{B}(M)\cap\operatorname{Spec}%
(B/I)$, where we consider $M$ and $M/IM$ as $B$-modules. By Lemma
\ref{lem:Nak} we get $\Psi_{B}(M)=\phi(\Psi(M))$ and thus
\begin{equation}
\Psi(M/IM)=\phi(\Psi(M))\cap\operatorname{Spec}(B/I).
\end{equation}

\subsection{Associated variety and annihilator variety}

\label{subsec:ASVars}

In this section we let $\mathfrak{q}$ be an arbitrary finite dimensional
complex Lie algebra, and let $U=U\left(  \mathfrak{q}\right)  $ be its
universal enveloping algebra with the usual increasing filtration $U^{i}%
,i\geq0$. By the PBW theorem the associated graded algebra $\overline{U}$ is
isomorphic to the symmetric algebra $S\left(  \mathfrak{q}\right)  $. For a
$\mathfrak{q}$-module $V$, we define its\emph{ annihilator} and
\emph{annihilator variety} as follows
\[
\operatorname*{Ann}\left(  V\right)  =\left\{  u\in U:\forall x\in V\text{
}ux=0\right\}  \text{, }\mathrm{An}\mathcal{V}(V):=\operatorname*{Var}\left(
\overline{\operatorname*{Ann}V}\right)  \subset\mathfrak{q}^{\ast}%
\]
Here $\overline{\operatorname*{Ann}V}\subset S\left(  \mathfrak{q}\right)  $
denotes the associated graded space of $\operatorname*{Ann}V$ under the
filtration inherited from $U$.

For the rest of the section we assume that $V$ is generated by a \emph{finite}
dimensional subspace $V^{0}$.

In this case we get a filtration $V^{i}=U^{i}V^{0}$. The associated graded
space $\overline{V}$ is then an $S\left(  \mathfrak{q}\right)  $ module and we
define the \emph{associated variety} to be
\[
\mathrm{As}\mathcal{V}(V):=\operatorname*{Var}(\operatorname*{Ann}\overline
{V})=\operatorname*{Supp}\left(  \overline{V}\right)  \subset\mathfrak{q}%
^{\ast}%
\]
It is standard that $\mathrm{As}\mathcal{V}(V)$ does not depend on the choice
of the generating subspace. More generally a filtration on $V$ is called a
\emph{good filtration} if the associated graded space is a finitely generated
$S\left(  \mathfrak{q}\right)  $-module, and any two good filtrations lead to
the same associated variety. For a submodule $W\subset V$, a good filtration
on $V$ induces good filtrations on $W$ and on $V/W$ by $W^{i}=W\cap V^{i}$ and
$(V/W)^{i}=V^{i}/W^{i}$.

\begin{lemma}
\label{lem:AVBer} \cite{Ber} $\mathrm{As}\mathcal{V}(V)=\operatorname*{Var}%
\left(  \overline{\operatorname*{Ann}V^{0}}\right)  $ where
$\operatorname*{Ann}V^{0}=\left\{  u\in U:\forall x\in V^{0}\text{
}ux=0\right\}  .$
\end{lemma}

\begin{corollary}
We have $\mathrm{As}\mathcal{V}(V)\subseteq\mathrm{An}\mathcal{V}(V)$, and
equality holds if $\mathfrak{q}$ is commutative.
\end{corollary}

If there is possibility of confusion we will write $\mathrm{As}\mathcal{V}%
_{{\mathfrak{q}}}\left(  V\right)  $ etc. to emphasise dependence on
$\mathfrak{q}$.

If ${\mathfrak{w}}$ is a subalgebra of {$\mathfrak{q}$} we define the
\emph{coinvariant} \emph{space} to be the quotient
\[
C_{{\mathfrak{w}}}\left(  V\right)  =V/{\mathfrak{w}}V
\]
If {$\mathfrak{w}$} is an ideal then $C_{{\mathfrak{w}}}V$ is a {$\mathfrak{q}%
$}-module and the action descends to the quotient Lie algebra $\mathfrak{r}%
:=${$\mathfrak{q}$}$/${$\mathfrak{w}$}.

\begin{lem}
\label{lem:Quot}If $V$ is a finitely generated $\mathfrak{q}$-module then
$\mathrm{As}\mathcal{V}_{\mathfrak{q}}(C_{{\mathfrak{w}}}V)=\mathrm{As}%
\mathcal{V}_{\mathfrak{r}}(C_{{\mathfrak{w}}}V)\subset\mathrm{As}%
\mathcal{V}_{{\mathfrak{q}}}(V)\cap\mathfrak{r}^{\ast}$, where $\mathfrak{r}%
^{\ast}\subset\mathfrak{q}^{\ast}$ in the usual way.
\end{lem}

\begin{proof}
Let $Y^{0}$ denote the image of the generating space $V^{0}$ under the
quotient map $V\rightarrow C_{{\mathfrak{w}}}\left(  V\right)  $. Then $Y^{0}$
generates $C_{{\mathfrak{w}}}\left(  V\right)  $ and $\operatorname*{Ann}%
\nolimits_{\mathfrak{q}}Y^{0}\supset\operatorname*{Ann}\nolimits_{\mathfrak{q}%
}V^{0}+${$\mathfrak{w}$}. The result now follows from Lemma \ref{lem:AVBer}.
\end{proof}

As was discussed in the introduction, the converse inclusion is not true in general.

If $G$ is a real reductive group and $\left(  \pi,W\right)  \in\mathcal{M}(G)$ then $W^{HC}$ is dense in $W$ and
we can choose a finite dimensional $K$-invariant generating subspace of
$W^{HC}$. It follows that we have
\[
\mathrm{An}\mathcal{V}(\pi)=\mathrm{An}\mathcal{V}(\pi^{HC})\subset
\mathcal{N}\text{, }\quad \mathrm{As}\mathcal{V}(\pi^{HC})\subset\mathcal{N}%
_{\theta}%
\]
$\mathrm{and}$ the two varieties are unions of $G_{\mathbb{C}}$-orbits and
$K_{\mathbb{C}}$-orbits respectively. Moreover, one has the following theorem.

\begin{thm}[{\cite[Theorem 8.4]{Vog-Unip}}]\label{thm:Vog}
Let $\sigma\in \mathcal{HC}(G)$ be irreducible and $\cO$ be the dense nilpotent coadjoint orbit in $\mathrm{An}\cV(\sigma)$. Then
\begin{enumerate}[(a)]
\item $\mathrm{As}\cV(\sigma) \subset \mathrm{An}\cV(\sigma) \cap {\fk}^{\bot}$.
\item $\cO \cap  {\fk}^{\bot}$ is the union of a finite number of $K_{\C}$-orbits $\cO_1, \dots, \cO_r$, each of which has dimension equal to half of the dimension of $\cO$.
\item Some of the $\cO_i$ are contained in $\mathrm{As}\cV(\sigma)$; they are precisely the $K_{\C}$-orbits of maximal dimension in $\mathrm{As}\cV(\sigma)$.
\end{enumerate}
\end{thm}

\subsection{Restricted roots and parabolic subgroups}

\label{subsec:ResRoots}
%?? Dima
 The key results of this section are Proposition
\ref{prop:nondeg} and Lemma \ref{lem:non-deg-par}, which are rather straightforward for split groups. The main point of this section is to prove these results for quasi-split  groups.

%?? Dima
Recall that $H=TA$ denotes our fixed $\theta$-stable maximally split Cartan
subgroup. Let $\Sigma$ and $\Sigma_{0}$ denote the root systems of
${\mathfrak{h}}$ in ${\mathfrak{g}}$ and ${\mathfrak{a}}_{0}$ in
${\mathfrak{g}}_{0}$ respectively, and let $\mathfrak{g}^{\alpha}%
\subset\mathfrak{g}$ and $\mathfrak{g}_{0}^{\beta}\subset\mathfrak{g}_{0}$
denote the root spaces for $\alpha\in\Sigma$ and $\beta\in\Sigma_{0}$. For
$\alpha\in\Sigma$ let $\widetilde{\alpha}$ denote the restriction of $\alpha$
to $\mathfrak{a}_{0}$ then either $\widetilde{\alpha}=0$ or else
$\widetilde{\alpha}\in\Sigma_{0}$. Moreover for any $\beta\in\Sigma_{0}$ we
have
\[
\dim_{\mathbb{R}}\left(  \mathfrak{g}_{0}^{\beta}\right)  =\left\vert \left\{
\alpha\in\Sigma:\widetilde{\alpha}=\beta\right\}  \right\vert
\]
Every $\alpha\in$ $\Sigma$ is real-valued on ${\mathfrak{a}}_{0}$ and
imaginary-valued on ${\mathfrak{t}}_{0}$. The involution $\theta$ acts
naturally on $\Sigma$ and if $\alpha^{\prime}=-\theta\alpha$ then we have%
\[
\alpha^{\prime}|_{\mathfrak{a}_{0}}=\alpha|_{\mathfrak{a}_{0}},\alpha^{\prime
}|_{\mathfrak{t}_{0}}=-\alpha|_{\mathfrak{t}_{0}}%
\]

\begin{lemma}
Let $G$ be a real reductive group then the following are equivalent:

\begin{enumerate}
\item $G$ is quasi-split.

\item For all $\alpha\in\Sigma$ we have $\widetilde{\alpha}\neq0$.

\item For all $\beta\in\Sigma_{0},$ we have $\dim_{\mathbb{R}}\left(
\mathfrak{g}_{0}^{\beta}\right)  \leq2$.
\end{enumerate}
\end{lemma}

\begin{proof}
Since $G$ is quasi-split iff $\mathfrak{g}_{0}$ has a Borel subalgebra, the
lemma depends only on the Lie algebra $\mathfrak{g}_{0}$. Moreover it suffices
to prove the lemma for simple factors of $\mathfrak{g}_{0}$. The result is
obvious for split and complex factors, and by \cite{He} the other possible
simple quasi-split factors are of the form
\[%
\begin{tabular}
[c]{|l|l|l|l|l|}\hline
$\mathfrak{g}_{0}$ & $\mathfrak{su}_{l,l}$ & $\mathfrak{su}_{l,l+1}$ &
$\mathfrak{so}_{l,l+2}$ & $\mathfrak{e}_{6\left(  2\right)  }$\\\hline
Label & $AIII\left(  r=2l-1\right)  $ & $AIII\left(  r=2l\right)  $ &
$DI\left(  r=l+1\right)  $ & $EII$\\\hline
\end{tabular}
\]
Now the lemma can be checked using Table VI of \cite{He}, where (2) means that
there are no black dots in the Satake diagram, and (3) means that each of the
multiplicities $m_{\lambda}$ and $m_{2\lambda}$ is at most $2$.
\end{proof}

Since in this paper we suppose that $G$ is quasi-split, we obtain

\begin{corollary}
\label{cor:QuasSplit} If $\alpha\in\Sigma$ satisfies $\dim_{\mathbb{R}}\left(
\mathfrak{g}_{0}^{\widetilde{\alpha}}\right)  =2$, then $\alpha|_{\mathfrak{t}%
_{0}}\neq0$ and $\mathfrak{g}^{\alpha}\cap\mathfrak{g}_{0}^{\widetilde{\alpha
}}=\left\{  0\right\}  $.
\end{corollary}

\begin{proof}
Suppose by way of contradiction that $\alpha|_{\mathfrak{t}_{0}}=0$. Since
$\dim_{\mathbb{R}}\left(  \mathfrak{g}_{0}^{\widetilde{\alpha}}\right)  =2$
there is a root $\alpha_{1}\neq\alpha$ such that $\alpha|_{\mathfrak{a}_{0}%
}=\alpha_{1}|_{\mathfrak{a}_{0}}$. Since $\alpha|_{\mathfrak{t}_{0}}=0$,
$\alpha_{1}$ must be nonzero on $\mathfrak{t}_{0}$, and thus $\alpha$,
$\alpha_{1}$ and $\alpha_{2}=-\theta\alpha_{1}$ are \emph{three} distinct
roots with the same restriction $\widetilde{\alpha}$, contrary to assumption.
Hence $\alpha|_{\mathfrak{t}_{0}}\neq0$.

Since $\alpha|_{\mathfrak{t}_{0}}$ is imaginary valued we may choose
$X\in\mathfrak{t}_{0}$ such that $\alpha\left(  X\right)  =i. $ Now suppose
$v\in\mathfrak{g}^{\alpha}\cap\mathfrak{g}_{0}^{\widetilde{\alpha}}$. Then we
have $\left[  X,v\right]  \in\left[  \mathfrak{t}_{0},\mathfrak{g}_{0}\right]
\subset\mathfrak{g}_{0}$ while on the other hand $\left[  X,v\right]
=\alpha\left(  X\right)  v=iv\in i\mathfrak{g}_{0}$. Thus we get $iv=0$, and
since $v$ was arbitrary we conclude that $\mathfrak{g}^{\alpha}\cap
\mathfrak{g}_{0}^{\widetilde{\alpha}}=\left\{  0\right\}  .$
\end{proof}

Our choice of $B$ determines simple roots $\Pi\subset\Sigma$ and $\Pi
_{0}\subset\Sigma_{0}$, and the restriction $\widetilde{\alpha}$ is simple if
$\alpha$ is simple. Let $\mathfrak{v},\Psi,\Psi_{0}$ be as before and define
${\mathfrak{v}}_{0}={\mathfrak{n}}_{0}/\left[  {\mathfrak{n}}_{0}%
,{\mathfrak{n}}_{0}\right]  $ so that ${\mathfrak{v}}=\left(  {\mathfrak{v}%
}_{0}\right)  _{\mathbb{C}}$ and $\Psi_{0}=i\mathfrak{v}_{0}^{\ast}$, where
$\mathfrak{v}_{0}^{\ast}$ denotes the space of $\mathbb{R}$-linear functionals
on ${\mathfrak{v}}_{0}$. The natural projection $\mathfrak{n}\rightarrow
\mathfrak{v}$ restricts to isomorphisms%
\[
\bigoplus_{\alpha\in\Pi} \left(  \mathfrak{g}^{\alpha}\right)  \cong%
\mathfrak{v},\bigoplus_{\beta\in\Pi_{0}} \left(  \mathfrak{g}_{0}^{\beta
}\right)  \cong\mathfrak{v}_{0}%
\]
We write $\mathfrak{z}\subset\mathfrak{h}$ for the center of $\mathfrak{g}$,
and $\mathfrak{z}\subset\mathfrak{s}_{\psi}\subset\mathfrak{h}$ for the
stabilizer of $\psi\in\Psi$. We recall that $\psi$ is said to be non-degenerate
if its $H_{\mathbb{C}}$-orbit is open.

\begin{lemma}
\label{lem:psi-nondeg}For $\psi\in\Psi$ the following are equivalent

\begin{enumerate}
\item $\psi$ is non-degenerate.

\item $\mathfrak{s}_{\psi}=\mathfrak{z}$.

\item $\psi|_{\mathfrak{g}^{\alpha}}\neq0$ for all $\alpha\in\Pi$.
\end{enumerate}
\end{lemma}

\begin{proof}
We note that $\dim\left(  \Psi\right)  =\dim\left(  \mathfrak{h/z}\right)
=\left\vert \Pi\right\vert $ is the semisimple rank of $G$, while the
dimension of the $H_{\mathbb{C}}$-orbit of $\psi$ is $\dim\left(
\mathfrak{h/s}_{\psi}\right)  $; thus (1) is equivalent to (2). Also we have
$X\in\mathfrak{z}$ iff $\alpha\left(  X\right)  =0$ for all $\alpha\in\Pi,$
while $X\in\mathfrak{s}_{\psi}$ iff $\alpha\left(  X\right)  =0$ whenever
$\psi|_{\mathfrak{g}^{\alpha}}\neq0$; thus (2) is equivalent to (3).
\end{proof}

We now prove an analogous characterization for $\psi\in\Psi_{0},$ using the
following elementary result.

\begin{lemma}
Let $W_{0}$ be a two-dimensional real vector space with complexification $W$,
and let $\omega$ be a $\mathbb{C}$-linear functional on $W$ that is real
valued on $W_{0}$. Let $W_{1}\subset W$ be a two dimensional real subspace
such that $W_{0}\cap W_{1}=0$, then we have
\[
\omega|_{W_{0}}=0\iff\omega=0\iff\omega|_{W_{1}}=0
\]

\end{lemma}

\begin{proof}
The first equivalence holds since $\omega$ is $\mathbb{C}$-linear. Also
clearly $\omega=0\implies\omega|_{W_{1}}=0$. Conversely suppose $\omega
|_{W_{1}}=0$. Since $\omega$ is real-valued on $W_{0}$ and $\dim_{\mathbb{R}%
}W_{0}=2$ we have $\ker$ $\omega\cap W_{0}$ $\neq0.$ Since $W_{0}\cap W_{1}%
=0$, this forces $\ker$ $\omega\supsetneq W_{1}$. Since $\dim_{\mathbb{C}}W=2$
we get $\omega=0$ as desired.
\end{proof}

\begin{proposition}
\label{prop:nondeg}For $\psi\in\Psi_{0}$ the following are equivalent

\begin{enumerate}
\item $\psi$ is non-degenerate.

\item $\psi|_{\mathfrak{g}_{0}^{\beta}}\neq0$ for all $\beta\in\Pi_{0}$.

\item The $H$ orbit of $\psi$ is open in $\Psi_{0}.$

\item $\mathfrak{s}_{\psi}\cap\mathfrak{h}_{0}=\mathfrak{z}\cap\mathfrak{g}%
_{0}.$
\end{enumerate}
\end{proposition}

\begin{proof}
The equivalence of (3) and (4) follows from a dimension argument similar to
Lemma \ref{lem:psi-nondeg}. It suffices to show that (4) is equivalent to (2)
of Lemma \ref{lem:psi-nondeg}, which is obvious, and that (2) is equivalent to
(3) of Lemma \ref{lem:psi-nondeg}. For the latter it is enough to show that if
$\alpha\in\Pi$ and $\beta=\widetilde{\alpha}$ then
\begin{equation}
\psi|_{\mathfrak{g}^{\alpha}}=0\iff\psi|_{\mathfrak{g}_{0}^{\beta}}=0
\label{=alphabeta}%
\end{equation}

Now (\ref{=alphabeta}) is obvious if $\dim\mathfrak{g}_{0}^{\beta}=1$ for then
$\mathfrak{g}^{\alpha}$ is the complexification of $\mathfrak{g}_{0}^{\beta}$.
Otherwise by Corollary \ref{cor:QuasSplit} we have $\mathfrak{g}_{0}^{\beta
}\cap\mathfrak{g}^{\alpha}=0$, and (\ref{=alphabeta}) follows from the
previous lemma with $W_{0}=\mathfrak{g}_{0}^{\beta},W_{1}=\mathfrak{g}%
^{\alpha},\omega=i\psi$.
\end{proof}

The standard parabolic subgroups of $G$ are those that contain $B$, and these
correspond bijectively to subsets of $\Pi_{0}$. Indeed every $P\supset B$
admits a Levi decomposition $P=LU$ with $\theta$-stable Levi component
$L\supset H$, the group $B\cap L$ is a Borel subgroup of $L$ and the
corresponding simple roots for ${\mathfrak{a}}_{0}$ in $\mathfrak{l}_{0}$ give
the desired subset of $\Pi_{0}$.

\begin{lemma}
\label{lem:non-deg-par}For $\psi\in\Psi_{0}$ there exists a standard parabolic
subgroup $P=LU$ such that $\psi$ vanishes on $\mathfrak{u}$ and restricts
to a non-degenerate character of $\mathfrak{l}\cap\mathfrak{n}$.
\end{lemma}

\begin{proof}
Let $P$ correspond to the set $\left\{  \beta\in\Pi_{0}:\psi|_{\mathfrak{g}%
_{0}^{\beta}}\neq0\right\}  $, then the result follows from Proposition
\ref{prop:nondeg}.
\end{proof}

\subsection{The Osborne lemma}

\label{subsec:JacFun}

Let $S({\mathfrak{g}})^{i}$ and $U({\mathfrak{g}})^{i}$ denote the usual
filtrations of the symmetric and enveloping algebras of ${\mathfrak{g}}$ and
let $I\left(  {\mathfrak{g}}\right)  =S\left(  {\mathfrak{g}}\right)
^{{\mathfrak{g}}}$ and $Z({\mathfrak{g}})=U\left(  {\mathfrak{g}}\right)
^{{\mathfrak{g}}}$ denote the subrings of ${\mathfrak{g}}$-invariants.

\begin{lemma}[{\cite[\S 3.7]{Wal1}}]
\label{prop:Osborne} There exist finite dimensional
subspaces $E\subset S({\mathfrak{g}})$, $F\subset U({\mathfrak{g}})$ such that%
\[
S({\mathfrak{g}})^{i}\subset S(\mathfrak{n})^{i}EI\left(  {\mathfrak{g}%
}\right)  S({\mathfrak{k}})\text{, }U({\mathfrak{g}})^{i}\subset
U(\mathfrak{n})^{i}FZ({\mathfrak{g}})U({\mathfrak{k}})
\]

\end{lemma}

As before let $\mathcal{N}\subset{\mathfrak{g}}^{\ast}$ be the null cone and
let $\mathcal{N}_{\theta}=\mathcal{N}\cap{\mathfrak{k}}^{\bot}$.

\begin{cor}
\label{cor:PrFinite} The projection $pr_{\mathfrak{n}^{\ast}}:\mathcal{N}%
_{\theta}\rightarrow\mathfrak{n}^{\ast}$ is a finite morphism.
\end{cor}

\begin{proof}
The maximal ideals $S^{>0}({\mathfrak{k}})\subset S({\mathfrak{k}})$ and
$I^{>0}({\mathfrak{g}})\subset I({\mathfrak{g}})$ vanish on on ${\mathfrak{k}%
}^{\bot}$ and $\mathcal{N}$ respectively, hence both ideals vanish on
$\mathcal{N}_{\theta}$. By Lemma \ref{prop:Osborne} ${\mathbb{C}%
}[\mathcal{N}_{\theta}]$ is generated by $E$ as a module over $S(\mathfrak{n}%
)={\mathbb{C}}[\mathfrak{n}^{\ast}]$.
\end{proof}

%\begin{lemma}
%Let $X,Y$ be irreducible affine varieties of the same dimension and let $f:X\rightarrow Y$
%be a finite morphism, then $f$ is surjective.
%\end{lemma}

%\begin{proof}
%Let $Z=f\left(  X\right)  \subset Y$. Since $f$ is finite, $Z$ is closed in
%$Y$ and moreover%
%\[
%\dim\left(  Z\right)  =\dim\left(  X\right)  =\dim\left(  Y\right)
%\]
%Since $Y$ is irreducible we get $Z=Y$.
%\end{proof}

\begin{corollary}
\label{cor:PrOnto} If $Z$ is any irreducible component of $\mathcal{N}%
_{\theta}$ then $pr_{\mathfrak{n}^{\ast}}\left(  Z\right)  =\mathfrak{n}%
^{\ast}$.
\end{corollary}

\begin{proof}
By Corollary \ref{cor:PrFinite}, $pr_{\mathfrak{n}^{\ast}}$ is a finite map
and thus its image is a closed subset of $\mathfrak{n}^{\ast}$ of the same
dimension as $Z$. By Theorem \ref{thm:Vog} $\dim Z=1/2\dim\left(
\mathcal{N}\right)  =\dim\left(  \mathfrak{n}^{\ast}\right)  $, thus
$pr_{\mathfrak{n}^{\ast}}\left(  Z\right)  $ has full dimension, so
$pr_{\mathfrak{n}^{\ast}}\left(  Z\right)  =\mathfrak{n}^{\ast}$.
\end{proof}

\begin{cor}
[Casselman-Osborne-Gabber]\label{cor:Osborne}
%Let $\fn$ denote the nilradical of the complexified Lie algebra of the minimal parabolic subgroup of $G$.
If $\sigma\in\mathcal{HC}(G)$ then $\sigma$ is finitely generated as a
$U(\mathfrak{n})$-module. Moreover, any good $\mathfrak{g}$-filtration on
$\sigma$ is good as an $\mathfrak{n}$-filtration, and every good
$\mathfrak{b}$-filtration on $\sigma$ is good as an $\mathfrak{n}$-filtration.
In particular, $\mathrm{As}\mathcal{V}_{\mathfrak{n}}(\sigma)=pr_{\mathfrak{n}%
^{\ast}}(\mathrm{As}\mathcal{V}(\sigma))$.
\end{cor}

For proof of the \textquotedblleft moreover\textquotedblright\ part see
\cite[\S 7.8.1]{JosLect} or \cite[Appendix B]{AGS}.

\subsection{The Casselman-Jacquet Functor}

\label{subsec:CasJac}

As before let $\mathfrak{b=h+n}$ be the Borel subalgebra of $\mathfrak{g}$.
For a $\mathfrak{b}$-module $V$ we define its $\mathfrak{n}$-adic completion
and its Jacquet module as follows:%
\[
\widehat{V}=\widehat{V}_{\mathfrak{n}}:=\lim\limits_{\longleftarrow
}V/\mathfrak{n}^{i}V,\quad JV= J\left(  V\right)  =J_{\mathfrak{b}}\left(
V\right)  :=\left(  \widehat{V}_{\mathfrak{n}}\right)  ^{\mathfrak{h}%
\text{-finite}}%
\]
We note that $J\left(  V\right)  $ is different from the Casselman-Jacquet
module considered in \cite{Wal1}. However it is closely related to the geometric
Jacquet functor considered in \cite{ENV} (see Theorem \ref{th:ENV} below).

Let $\mathcal{G}\left(  \mathfrak{b}\right)  $ be the category of finitely
generated $\mathfrak{b}$-modules for which every good $\mathfrak{b}%
$-filtration is also good as an $\mathfrak{n}$-filtration. Note that
$\mathcal{G}\left(  \mathfrak{b}\right)  $ is closed under subquotients. The
following result is due to Gabber.

\begin{theorem}[{\cite[\S 7]{JosLect}}]
\label{thm:CassJacEmb}If $V\in\mathcal{G}%
(\mathfrak{b})$ then we have $\bigcap_{k\geq0}\mathfrak{n}^{k}V=0$. Hence $V$
embeds into $\widehat{V}$ with dense image.
\end{theorem}

By the Artin-Rees theorem for nilpotent Lie algebras (\cite[Theorem 4.2]{McC})
we deduce

\begin{corollary}
\label{cor:vfaith} $V\mapsto\widehat{V}$ is an exact faithful functor from
$\mathcal{G}(\mathfrak{b})$ to the category of $\mathfrak{b}$-modules.
\end{corollary}

%?? Dima
An analogous statement for $V\in\mathcal{HC}(G)$ was first proven by Casselman
(see \cite{Cas}). By Corollary \ref{cor:Osborne}, $\mathcal{HC}(G)$ naturally
embeds into $\mathcal{G}(\mathfrak{b})$ and thus Casselman's theorem is a
special case of Corollary \ref{cor:vfaith}.
%?? Dima

\begin{lemma}
[{\cite[\S 3.5]{JosVar}}]\label{lem:Sinf} If $V\in\mathcal{G}(\mathfrak{b})$
then there exists a \emph{finite dimensional} $\mathfrak{h}$-invariant
subspace $S_{\infty}\subset J\left(  V\right)  $, which maps onto
$V/\mathfrak{n}V.$
\end{lemma}

Since this result plays a key role in the subsequent discussion, we include a
proof here.

\begin{proof}
Let $\Omega_{j}$ be the set of (generalized) weights of ${\mathfrak{h}}$
appearing in $\mathfrak{n}^{j}V/\mathfrak{n}^{j+1}V$. Since the action of
$\mathfrak{n}$ shifts the weights in the positive direction, there exists $i$
such that $\Omega_{j}\cap\Omega_{0}=\emptyset$ for all $j\geq i$. Let us
define
\[
\overline{S}=\bigoplus_{\mu\in\Omega_{0}}(V/\mathfrak{n}^{i}V)^{\mu}.
\]
Then any (generalized) ${\mathfrak{h}}$ -eigenvector of $\overline{S}$ can be
lifted by successive approximation to a (generalized) ${\mathfrak{h}}$
-eigenvector of the same weight in $\hat{V}$. In this way we find an
${\mathfrak{h}}$-invariant finite dimensional subspace $S_{\infty}\subset
\hat{V}$ that maps bijectively to $\overline{S}$ and thus onto $V/\mathfrak{n}%
V.$
\end{proof}

\begin{lemma}
\label{lem:Wdense}If $V\in\mathcal{G}(\mathfrak{b})$ and $W\subset J\left(
V\right)  $ is a dense $\mathfrak{h}$-submodule of $\widehat{V}$ then
$W=J\left(  V\right)  .$
\end{lemma}

\begin{proof}
Note that for any $i$, the natural projection defines an isomorphism
$\widehat{V}/\mathfrak{n} ^{i}\widehat{V} \cong V/\mathfrak{n}^{i}V$, with the
inverse given by the natural map $V \to\widehat{V}$. Note also that the
density of $W$ implies that $W$ projects onto $V/\mathfrak{n}^{i}V$ for any
$i$.

Now let $J\left(  V\right)  ^{\mu}=\left(  \widehat{V}\right)  ^{\mu}$ be the
generalized $\mathfrak{h}$-eigenspace for some fixed weight $\mu$. Then for
all sufficiently large $i$ we have $J\left(  V\right)  ^{\mu}\cap
\mathfrak{n}^{i}\widehat{V}=0$ and thus
\begin{equation}
J\left(  V\right)  ^{\mu}\cong\left(  \widehat{V}/\mathfrak{n}^{i}%
\widehat{V}\right)  ^{\mu}\cong\left(  V/\mathfrak{n}^{i}V\right)  ^{\mu}\cong
W^{\mu} \label{=JVmu}%
\end{equation}

This implies $J\left(  V\right)  =W.$
%??Dima

\end{proof}

\begin{lemma}
\label{lem:Jdense}$J\left(  V\right)  $ is dense in $\widehat{V}$ for any
$V\in\mathcal{G}(\mathfrak{b})$. Moreover $V\mapsto J\left(  V\right)  $ is an
exact faithful functor from $\mathcal{G}(\mathfrak{b})$ to $\mathcal{G}%
(\mathfrak{b})$.
\end{lemma}

\begin{proof}
Let $S_{\infty}$ be as in Lemma \ref{lem:Sinf} and let $V_{\infty}\subset
J\left(  V\right)  $ be the $\mathfrak{n}$-submodule generated by $S_{\infty}%
$, then it follows that $V_{\infty}\in\mathcal{G}(\mathfrak{b})$. Also arguing
by induction on $i$ we deduce that $V_{\infty}$ surjects onto each
$\mathfrak{n}^{i}V/\mathfrak{n}^{i+1}V$ and hence that $V_{\infty}$ is dense
in $\widehat{V}$. By Lemma \ref{lem:Wdense} $J\left(  V\right)  =V_{\infty}$ ,
and thus $J\left(  V\right)  $ is dense and belongs to $\mathcal{G}%
(\mathfrak{b})$.

Corollary \ref{cor:vfaith} implies that $J$ is left exact. For right
exactness, we need to show that if $\phi:V\rightarrow V^{\prime}$ is a
surjection then so is $J\phi$; since the image of $J\phi$ is dense in
$\widehat{V^{\prime}}$, this follows from Lemma \ref{lem:Wdense}. Now to prove
faithfulness it suffices to show $V\neq0$ implies $J\left(  V\right)  \neq0$,
but this follows from Corollary \ref{cor:vfaith} and the density of $J\left(
V\right)  $ in $\widehat{V}$.
\end{proof}

If $M\in\mathcal{HC}\left(  {\mathfrak{g}},K\right)  $ then $M\in
\mathcal{G}\left(  \mathfrak{b}\right)  $ by Corollary \ref{cor:Osborne} so
the above results apply to $M$, indeed in this case Corollary \ref{cor:vfaith}
is due to Casselman. However one can say more. Let $\bar{B}=\theta\left(
B\right)  $ be the opposite Borel subgroup, and let $\mathcal{C}\left(
{\mathfrak{g}},{\mathfrak{\bar{b}}}\right)  $ be the category of finitely
generated ${\mathfrak{g}}$-modules, which are ${\mathfrak{\bar{b}}}$-finite.

\begin{theorem}
\label{th:ENV}If $M\in\mathcal{HC}\left(  {\mathfrak{g}},K\right)  $ then
$\widehat{M}$ is a ${\mathfrak{g}}$-module and we have

\begin{enumerate}[(a)]
\item \label{it:Trans} $JM=\left(  \widehat{M}_{\mathfrak{n}}\right)
^{\mathfrak{\bar{n}}\text{-finite}}$.

\item \label{it:Ocat} $JM\in\mathcal{C}\left(  {\mathfrak{g}},{\mathfrak{\bar
{b}}}\right)  .$
\end{enumerate}
\end{theorem}

\begin{proof}
Part (\ref{it:Trans}) follows from \cite[Proposition 2.4]{ENV}. More precisely
\cite{ENV} proves
\[
\left(  \widehat{M}_{\mathfrak{\bar{n}}}\right)  ^{\mathfrak{n}\text{-finite}%
}=\left(  \widehat{M}_{\mathfrak{\bar{n}}}\right)  ^{\mathfrak{h}%
\text{-finite}}%
\]
and we get part (\ref{it:Trans}) upon replacing $\mathfrak{n}$ by
$\mathfrak{\bar{n}}$.

%?? Dima
For part (\ref{it:Ocat}), note that $J(M)$ is locally ${\mathfrak{h}}$-finite
by definition, locally $\overline{\mathfrak{n}}$-finite by part
(\ref{it:Trans}) and finitely generated over $\mathfrak{g}$ by Lemma
\ref{lem:Jdense}.
%?? Dima

\end{proof}

%?? Dima
The theorem implies that $\mathrm{As}\mathcal{V}_{{\mathfrak{g}}%
}(JM)=\mathrm{As}\mathcal{V}_{\mathfrak{n}}(JM)$ and thus from now on we will
write just $\mathrm{As}\mathcal{V}(JM)$.

\begin{remark}
Theorem \ref{th:ENV} implies that $\mathrm{As}\mathcal{V}(JM)$ is a union of
$\overline{B}$-orbits in $\mathrm{An}\mathcal{V}(JM)\cap\overline
{\mathfrak{b}}^{\bot}$. Theorem \ref{thm:CassJacEmb} and Lemma
\ref{lem:Jdense} imply that $\mathrm{An}\mathcal{V}(JM)=\mathrm{An}%
\mathcal{V}(M) = \mathrm{An}\mathcal{V}(\widehat{M}) $ since both $JM$ and $M$
densely embed into $\widehat{M}$ and the action of $\mathfrak{g}$ on
$\widehat{M}$ is continuous.  It is also known that $\dim{\mathcal{O}}%
\cap\overline{\mathfrak{b}}^{\bot}=1/2\dim{\mathcal{O}}$, for any coadjoint
orbit ${\mathcal{O}} \subset\mathfrak{g}^{*}$, and that $\dim\mathrm{As}%
\mathcal{V}(V)\geq1/2 \dim\mathrm{An}\mathcal{V}(V)$, for any
finitely-generated module $V$ over any algebraic Lie algebra (see
\cite[Proposition 6.1.4]{JosLect}).

Altogether we obtain that for an irreducible $M$, the associated variety
$\mathrm{As}\mathcal{V}(JM)$ is a union of irreducible components of maximal
dimension in $\mathrm{An}\mathcal{V}(M)\cap\overline{\mathfrak{b}}^{\bot}$.
Unfortunately, this does not determine $\mathrm{As}\mathcal{V}(JM)$ since the
variety $\mathrm{An}\mathcal{V}(M)\cap\overline{\mathfrak{b}}^{\bot}$ has many
irreducible components.
\end{remark}

%?? Dima

\subsection{Whittaker Functionals}

We recall that a representation in $\mathcal{M}$ or $\mathcal{HC}$ is said to
be \emph{large} if its annihilator variety is the nilpotent cone
$\mathcal{N}({\mathfrak{g}})$, and \emph{generic} if it admits a Whittaker
functional for some non-degenerate $\psi\in\Psi$.

\begin{theorem}
\label{thm:KosCHM} For $\pi\in\mathcal{M}$ the following are equivalent:
\begin{equation}
\label{=koswhit}\pi\text{ is generic} \Leftrightarrow\pi\text{ is
large}\Leftrightarrow\pi^{HC}\text{ is large} \Leftrightarrow\pi^{HC}\text{ is
generic.}%
\end{equation}

Moreover if $\pi$ is large and $\psi\in\Psi$ is non-degenerate, then

\begin{enumerate}[(a)]
\item \label{it:HCL}$Wh_{\psi}^{\prime}\left(  \pi^{HC}\right)  \neq0$.

\item \label{it:KosT} If $\psi\in\Psi_{0}$ then there exists $a\in F_{G}$ such
that $Wh_{a\cdot\psi}^{\ast}(\pi)\neq0$

\item \label{it:nonUn} If $\psi\not \in \Psi_{0}$ then $Wh_{\psi}^{\ast}%
(\pi)=0$.
\end{enumerate}
\end{theorem}

The equivalence $\pi\text{ is large}\Leftrightarrow\pi^{HC}\text{ is large}$
is obvious since $\pi^{HC}$ is dense in $\pi$ and thus they have the same
annihilator. For $\pi$ irreducible the other equivalences in (\ref{=koswhit})
are in \cite[Theorems K and L]{Kos}. Part (\ref{it:HCL}) follows from
\cite[Corollary 2.2.2]{MatActa}. Part (\ref{it:KosT}) is \cite[Theorem K]{Kos}
and Part (\ref{it:nonUn}) is in \cite[Theorem 8.2]{CHM}. The case of general
$\pi$ follows from this by exactness of the functors $Wh_{\psi}^{\ast}$ and
$Wh_{\psi}^{\prime}$ proved in \cite[Theorem 8.2]{CHM} and \cite[Theorem
4.3]{Kos} respectively.

\section{Proof of Theorem \ref{thm:Main}}

\label{sec:PfMain} \setcounter{lemma}{0} Let $\mathfrak{n}^{\prime
}=[\mathfrak{n},\mathfrak{n}]$ and $\mathfrak{v}=\mathfrak{n}/\mathfrak{n}%
^{\prime}$ be as in (\ref{=v}), and for an $\mathfrak{n}$-module $V$, we
denote the $\mathfrak{v}$-module of $\mathfrak{n}^{\prime}$-coinvariants by
\[
CV=C\left(  V\right)  =C_{\mathfrak{n}^{\prime}}\left(  V\right)
:=V/\mathfrak{n}^{\prime}V
\]
Since $\mathfrak{v}$ is commutative $\mathrm{An}\mathcal{V}_{\mathfrak{v}%
}(CV)=\mathrm{As}\mathcal{V}_{\mathfrak{v}}(CV)$ and we simply write
$\mathcal{V}_{\mathfrak{v}}(CV)$. We note that%
\[
V\in\mathcal{G}(\mathfrak{b})\implies C\left(  V\right)  \in\mathcal{G}%
(\mathfrak{b}).
\]
%as well.

For the rest of this section let $M\in\mathcal{HC}\left(  G\right)  $ denote a
fixed Harish-Chandra module.

\begin{lemma}
\label{lem:AVVSupp} We have
\[
\Psi(M)=\mathrm{Supp}_{\mathfrak{v}}(CM)=\mathcal{V}_{\mathfrak{v}}\left(
CM\right)  .
\]

\end{lemma}

\begin{proof}
The Lie algebra $\mathfrak{h}$ contains an element $h$ that acts by $1$ on
$\mathfrak{v}$, and by the degree on $S(\mathfrak{v})$. Since the ideal
$Ann_{S(\mathfrak{v})}(CM)$ is $\mathfrak{h}$-invariant, it is homogeneous and
consequently $\mathrm{Supp}_{\mathfrak{v}}(CM)=\mathcal{V}_{\mathfrak{v}}%
(CM)$. Finally $\Psi(M)=\mathrm{Supp}_{\mathfrak{v}}(CM)$ by Nakayama's lemma
(see \S \ref{subsec:Nakayama}).
\end{proof}

For the proof of Theorem \ref{thm:Main} we need three further results, which
are stated below and proved in sections \ref{subsec:PfQuotJac}, \ref{subsec:PfJacQuot}, \ref{subsec:PfAVJ}.

\begin{lemma}
\label{lem:QuotJac} We have a $\mathfrak{b}$-module isomorphism $C\left(
JM\right)  \approx J(CM)$
\end{lemma}

\begin{lemma}
\label{lem:JacQuot} $\mathcal{V}_{\mathfrak{v}}(C\left(  JM\right)
)=\mathrm{As}\mathcal{V}_{\mathfrak{n}}(JM)\cap\Psi$.
\end{lemma}

\begin{lemma}
\label{lem:AVJ}\textrm{ }$\mathrm{As}\mathcal{V}_{\mathfrak{n}}(JM)\supset
\mathrm{As}\mathcal{V}_{\mathfrak{n}}(M)\cap\Psi$.
\end{lemma}

We now prove Theorem \ref{thm:Main}.

\begin{proof}
[Proof of Theorem \ref{thm:Main}]By Lemma \ref{lem:AVVSupp} and Corollary
\ref{cor:Osborne} we have
\[
\Psi(M)=\mathcal{V}_{\mathfrak{v}}(CM),\mathrm{As}\mathcal{V}_{\mathfrak{n}%
}(M)=pr_{\mathfrak{n}^{\ast}}(\mathrm{As}\mathcal{V}_{{\mathfrak{g}}}(M))
\]
By Lemma \ref{lem:Quot} we have$\mathcal{V}_{\mathfrak{v}}(CM)\subset
\mathrm{As}\mathcal{V}_{\mathfrak{n}}(M)\cap\Psi$, and it remains only to
prove%
\begin{equation}
\mathcal{V}_{\mathfrak{v}}(CM)\supset\mathrm{As}\mathcal{V}_{\mathfrak{n}%
}(M)\cap\Psi\label{=remain}%
\end{equation}

By Corollary \ref{cor:Osborne} $M\in\mathcal{G}(\mathfrak{b})$ and hence
$CM\in\mathcal{G}(\mathfrak{b})$ as well. By Lemma \ref{lem:Jdense} $J\left(
CM\right)  $ is dense in $\widehat{CM}$, since $CM$ is also dense in
$\widehat{CM}$ we get
\[
\mathcal{V}_{\mathfrak{v}}\left(  CM\right)  =\mathrm{An}\mathcal{V}%
_{\mathfrak{v}}\left(  \widehat{CM}\right)  =\mathcal{V}_{\mathfrak{v}}\left(
J\left(  CM\right)  \right)
\]
Now by Lemmas \ref{lem:QuotJac}, \ref{lem:JacQuot} and \ref{lem:AVJ} we get
\[
\mathcal{V}_{\mathfrak{v}}(J\left(  CM\right)  )=\mathcal{V}_{\mathfrak{v}%
}\left(  C\left(  JM\right)  \right)  =\mathrm{As}\mathcal{V}_{\mathfrak{n}%
}(JM)\cap\Psi\supset\mathrm{As}\mathcal{V}_{\mathfrak{n}}(M)\cap\Psi\text{ }%
\]

This proves (\ref{=remain}) and finishes the proof of Theorem \ref{thm:Main}.
\end{proof}

\subsection{Proof of Lemma \ref{lem:QuotJac}}

\label{subsec:PfQuotJac}

For $V\in\mathcal{G}(\mathfrak{b})$ we let $\widehat{V}$ denote its
$\mathfrak{n}$-adic completion and let $J(V)=\left(  \widehat{V}\right)
^{{\mathfrak{h}}\text{-finite}}$ denote the associated Jacquet functor as
before. In this section we prove Lemma \ref{lem:QuotJac} in a more general
setting. Let $\mathfrak{c}\subset\mathfrak{n}$ be any ${\mathfrak{h}}%
$-invariant ideal, and define $C_{{\mathfrak{c}}}${$\left(  V\right)  =$}%
$V/${$\mathfrak{c}$}$V$.

\begin{lemma}
For $V\in\mathcal{G}(\mathfrak{b})$ we have $C_{{\mathfrak{c}}}J(V)\approx
J(C_{{\mathfrak{c}}}V)$.
\end{lemma}

\begin{proof}
By Lemma \ref{lem:Jdense} $J$ is exact hence it is enough to show that
{$\mathfrak{c}$}$J(V)=J({\mathfrak{c}}V)$ as submodules of $J(V)$. Since $V$
dense in $\widehat{V}$, ${\mathfrak{c}}\widehat{V}$ is contained in the
closure of ${\mathfrak{c}}V$ in $\widehat{V}$, which by the Artin-Rees theorem
(\cite[Theorem 4.2]{McC}) coincides with $\widehat{{\mathfrak{c}}V}$. Since
${\mathfrak{c}}\widehat{V}$ contains ${\mathfrak{c}}V$ we see that
${\mathfrak{c}}\widehat{V}$ is dense $\widehat{{\mathfrak{c}}V}$, and since
$J(V)$ is dense in $\widehat{V}$ it follows that ${\mathfrak{c}}J(V)$ is dense
in $\widehat{{\mathfrak{c}}V}$. Evidently ${\mathfrak{c}}J(V)\subset\left(
\widehat{{\mathfrak{c}}V}\right)  ^{{\mathfrak{h}}\text{-finite}%
}=J({\mathfrak{c}}V)$, hence {$\mathfrak{c}$}$J(V)=J({\mathfrak{c}}V)$ by
Lemma \ref{lem:Wdense}.
\end{proof}

\subsection{Proof of Lemma \ref{lem:JacQuot}}

\label{subsec:PfJacQuot}

We prove Lemma \ref{lem:JacQuot} for a more general class of modules. As
before, let
\[
\mathfrak{b=h+n,n}^{\prime}=\left[  \mathfrak{n},\mathfrak{n}\right]
,\mathfrak{v}=\mathfrak{n/n}^{\prime},\Psi=\mathfrak{v}^{\ast}.
\]
Let $\mathcal{J}\left(  \mathfrak{b}\right)  $ be the category of
$\mathfrak{b}$-modules with a finite dimensional $\mathfrak{h}$-invariant
generating subspace. Evidently if $M$ is a Harish-Chandra module, or even from
category $\mathcal{G}\left(  \mathfrak{b}\right)  $, then $JM\in$
$\mathcal{J}\left(  \mathfrak{b}\right)  $. Therefore Lemma \ref{lem:JacQuot}
follows from the next result.

\begin{lem}
\label{lem:LWQuot} If $V\in\mathcal{J}\left(  \mathfrak{b}\right)  $ then we
have $\mathcal{V}_{\mathfrak{v}}(C_{\mathfrak{n}^{\prime}}V)=\mathrm{As}%
\mathcal{V}_{\mathfrak{n}}(V)\cap\Psi$.
\end{lem}

\begin{proof}
By Lemma \ref{lem:Quot}, we have $\mathcal{V}_{\mathfrak{v}}(C_{\mathfrak{n}%
^{\prime}}V)=\mathrm{As}\mathcal{V}_{\mathfrak{n}}(C_{\mathfrak{n}^{\prime}%
}V)\subset\mathrm{As}\mathcal{V}_{\mathfrak{n}}(V)\cap\Psi$, and so it
suffices to prove the reverse containment. Let $E$ be a finite dimensional
$\mathfrak{h}$-invariant generating subspace of $V$, and let $F$ be its image
in $C_{\mathfrak{n}^{\prime}}\left(  V\right)  $. By Lemma \ref{lem:AVBer} we
have
\[
\mathrm{As}\mathcal{V}_{\mathfrak{n}}(C_{\mathfrak{n}^{\prime}}%
V)=\operatorname*{Var}\left(  \overline{J}\right)  ,
\]
where $J$ is the annihilator of $F$ in $U=U\left(  \mathfrak{n}\right)  $ and
$\overline{J}\subset S\left(  \mathfrak{n}\right)  $ is its associated graded
space under the usual filtration $U^{i}$ of $U$. Therefore it is enough to
prove that $\overline{J}$ vanishes on $\mathrm{As}\mathcal{V}_{\mathfrak{n}%
}(V)\cap\Psi$, i.e. that if $u\in J$ then $\overline{u}$ vanishes on
$\mathrm{As}\mathcal{V}_{\mathfrak{n}}(V)\cap\Psi$. To prove this we need some
additional notation.

We fix $\rho^{\vee}\in\mathfrak{h}$ satisfying $\alpha\left(  \rho^{\vee
}\right)  =1$ for every simple root $\alpha$, and for an $\mathfrak{h}$-module
$X$ we consider generalized $\rho^{\vee}$-weights, which we refer to simply as
\emph{weights}. We write $X_{\mu}$ for the $\mu$-weight space for $\mu
\in\mathbb{C}$, and if $x$ is a weight vector we write $\left[  x\right]  $
for the real part of its weight; thus $\left[  x\right]  =\operatorname{Re}%
\left(  \mu\right)  $ for $x\in X_{\mu}$. This notation will be applied to
$U,V$ and to the filtrands $U^{i}$ and $V^{i}=U^{i}E$. We also fix a weight
basis $v_{1},\dots,v_{m}$ of $E$, ordered so that $\left[  v_{i}\right]
\geq\left[  v_{j}\right]  $ if $i\geq j.$

If $u\in J$, then $u\in J\cap U^{d}$ for some $d$, and since
$J$ is $ad\left(  \mathfrak{h}\right)  $-stable we may assume $u\in U_{l}^{d}$
for some integer $l$. If $l>d$ then $u\in\mathfrak{n}^{\prime}U^{d-1}$ and
$\overline{u}=0$ on all of $\Psi$, therefore we may assume that $l\leq d$. For
$1\leq t\leq m$ let $L^{t}\subset V$ denote the submodule generated by
$v_{1},\dots,v_{t}$. Since $V$ is glued from the subquotients $L^{t}/L^{t-1}$
we have%
\[
\mathrm{As}\mathcal{V}_{\mathfrak{n}}(V)=\bigcup_{t}\mathrm{As}\mathcal{V}%
_{\mathfrak{n}}(L^{t}/L^{t-1})\text{.}%
\]
Thus it suffices to show that $u$ vanishes on $\mathrm{As}\mathcal{V}%
_{\mathfrak{n}}(L^{t}/L^{t-1})$ for each $t$, i.e. that%
\[
uv_{t}\in L^{t-1}+V^{d-1}.
\]

Now we may write $uv_{t}=\sum_{i=1}^{m}\left(  \sum_{j}X_{ij}b_{ij}%
v_{i}\right)  $, where $X_{ij}\in\mathfrak{n}^{\prime}$ and $b_{ij}\in U$ are
weight vectors satisfying
\[
\lbrack uv_{t}]=\left[  X_{ij}\right]  +[b_{ij}]+[v_{i}].
\]
We have $\left[  X_{ij}\right]  \geq2,$ $\left[  u\right]  =l\leq d$, and
$[v_{t}]\leq\lbrack v_{i}]$ for $i\geq t$. Thus we get
\[
\lbrack b_{ij}]=\left[  u\right]  -\left[  X_{ij}\right]  +[v_{t}]-[v_{i}]\leq
d-2\text{ for }i\geq t\text{.}%
\]
It follows that for $i\geq t$ we have $b_{ij}\in U^{d-2}$ and $X_{ij}b_{ij}\in
U^{d-1}$. Hence we get
\[
uv_{t}=\sum_{i=1}^{t-1}\sum_{j}X_{ij}b_{ij}v_{i}+\sum_{i=t}^{m}\sum_{j}%
X_{ij}b_{ij}v_{i}\in L^{t-1}+V^{d-1}\text{.}%
\]

\end{proof}

\subsection{Proof of Lemma \ref{lem:AVJ}}

\label{subsec:PfAVJ} \setcounter{lemma}{0}

We will use Beilinson-Bernstein localization \cite{BB}, the paper \cite{ENV}
that describes the Casselman-Jacquet functor in geometric terms, and the paper
\cite{Gin} that describes the behavior of the singular support of $D$-modules
under the nearby cycle functor. Let us describe the setting in detail.

Let $M$ be an admissible $\left(  \mathfrak{g},K\right)  $ module with
infinitesimal character $\chi_{\lambda}$, with parameter $\lambda$ chosen to
be dominant. We note that the action of $K$ can be complexified since it is locally finite. Then $M$ is a $\left(  U_{\lambda},K_{\C}\right)  $-module, where
$U_{\lambda}$ is the quotient of $U\left(  \mathfrak{g}\right)  $ by the
two-sided ideal generated by $z-\chi_{\lambda}\left(  z\right)  $.  Let
$\mathcal{D}_{\lambda}$ denote the $\lambda$-twisted sheaf of differential
operators on the flag variety $X$, then $U_{\lambda}=\Gamma\left(
X,\mathcal{D}_{\lambda}\right)  $. By a $\left(  \mathcal{D}_{\lambda
},K_{\C}\right)  $-module we mean a coherent $\mathcal{D}_{\lambda}$-module that is
$K_{\C}$-equivariant. Such a module is necessarily holonomic with regular
singularities. By Beilinson-Bernstein (\cite{BB}) the global sections functor
\[
\Gamma:\left\{  \left(  \mathcal{D}_{\lambda},K_{\C}\right)  \text{-modules}%
\right\}  \rightarrow\left\{  \left(  U_{\lambda},K_{\C}\right)  \text{-modules}%
\right\}
\]
is exact and essentially surjective, a section of $\Gamma$ is given by the
localization functor $\mathcal{D}_{\lambda}\otimes_{U_{\lambda}}\left(
\cdot\right)  $. Moreover if $\lambda$ is regular then $\Gamma$ is an
equivalence of categories. Let $X_{1},\ldots,X_{n}$ be the $K_{\C}$-orbits on $X$,
and let $T_{X_{i}}^{\ast}X$ denote the corresponding conormal bundles. If
$\mathcal{M}$ is a $\left(  \mathcal{D}_{\lambda},K_{\C}\right)  $-module, then its
characteristic cycle (see \cite{Gin}) is of the form%
\[
SS\left(  \mathcal{M}\right)  =\sum_{i=1}^{n}m_{i}T_{X_{i}}^{\ast}X.
\]
for some nonnegative integers $m_{i}$. The characteristic variety $CV\left(
\mathcal{M}\right)  $ is the union of $T_{X_{i}}^{\ast}X$ for which $m_{i}>0$.
Let us describe the connection between the characteristic cycle of a
$\mathcal{D}_{\lambda}$-module $\mathcal{M}$ and the associated cycle of the
Harish-Chandra module $M:=\Gamma(\mathcal{M})$. Any point $x\in X$ defines a
Borel subalgebra $\mathfrak{b}_{x}\subset{\mathfrak{g}}$. The tangent space
$T_{x}X$ can be identified with ${\mathfrak{g}}/\mathfrak{b}_{x}$ and the
cotangent space with $({\mathfrak{g}}/\mathfrak{b}_{x})^{\ast}=(\mathfrak{b}%
_{x})^{\bot}\subset{\mathfrak{g}}^{\ast}$. This gives a natural embedding of
the cotangent bundle $T^{\ast}X$ into the trivial bundle $X\times
{\mathfrak{g}}^{\ast}$. The composition of this map with the projection on the
second coordinate is called the moment map, denoted by $\mu$. By a result of
Borho and Brylinski (\cite{BorBry}) we have
\begin{equation}
\mu(CV(\mathcal{M}))=\mathrm{As}\mathcal{V}_{{\mathfrak{g}}}(M)
\label{eq:CVAV}%
\end{equation}

By Corollary \ref{cor:Osborne} we have%

\begin{equation}
\mathrm{As}\mathcal{V}_{\mathfrak{n}}(M)=pr_{\mathfrak{n}^{*}}(\mathrm{As}%
\mathcal{V}_{{\mathfrak{g}}}(M)) \label{eq:AVng}%
\end{equation}

The paper \cite{ENV} gives a precise geometric description of $J\left(
M\right)  $, which we now recall briefly. Actually \cite{ENV} deals with
$J_{\mathfrak{\bar{n}}}\left(  M\right)  $, so the description below is a
trivial modification of \cite{ENV}. Let $H$ be the maximally split torus of
$G$ and let $\rho^{\vee}:\mathbb{G}_{m}\rightarrow H_{\C}$ be the cocharacter such
that $\alpha\circ d\rho^{\vee}=-Id_{\mathbb{C}}$ for every simple root
$\alpha$. By composing $\rho^{\vee}$ with the action of $G_{\C}$ on $X,$ we get an
action map $a:\mathbb{G}_{m}\times X\rightarrow X$. Consider now the following
diagram%
\[
X\overset{a}{\longleftarrow}\mathbb{G}_{m}\times X\overset{j}{\longrightarrow
}\mathbb{A}^{1}\times X\overset{i}{\longleftarrow}\left\{  0\right\}  \times
X\approx X
\]
For a $\left(  \mathcal{D}_{\lambda},K_{\C}\right)  $ module $\mathcal{M}$, let
$\Phi\left(  \mathcal{M}\right)  $ be the $\mathcal{D}_{\lambda}$-module
obtained by applying the nearby cycles functor to $j_{\ast}a^{\ast}\left(
\mathcal{M}\right)  $ along $\left\{  0\right\}  \times X\approx X.$

\begin{theorem}
\label{thm:GeoJac} \cite{ENV} $\Phi\left(  \mathcal{M}\right)  $ is a $\left(
\mathcal{D}_{\lambda},\bar{N_{\C}}\right)  $-module and one has%
\[
\Gamma\left(  \Phi\left(  \mathcal{M}\right)  \right)  =J_{\mathfrak{n}%
}\left(  \Gamma\left(  \mathcal{M}\right)  \right)  .
\]

\end{theorem}

In view of this theorem $\Phi\left(  \mathcal{M}\right)  $ can be regarded as
the geometric Casselman-Jacquet functor.

The paper \cite{Gin} describes the behavior of the characteristic cycle under
the nearby cycle functor in the following way. For an algebraic variety $Z$,
and a regular function $f:Z \to{\mathbb{C}}$ let $U:=f^{-1}({\mathbb{C}}
\setminus\{0\})$. Suppose we have an algebraic family $S_{t}$ of subvarieties
of $Z_{t}:=f^{-1}(t)$ parameterized by $t \in{\mathbb{C}} \setminus\{0\}$. Let
$S \subset U$ denote the total space of this family and let $\overline{S}$
denote the closure of $S$ in $Z$. Denote by $\lim_{t \to0} S_{t}$ the
algebraic cycle corresponding to the scheme-theoretic intersection
$\overline{S} \cap f^{-1}(0)$ (cf. \cite[1.4]{Gin}).

\begin{thm}
[\cite{Gin}, Theorem 5.5]\label{thm:CVnearby} Let $\mathcal{M}$ be a holonomic
$\mathcal{D}_{\lambda}$-module with over $Z$ with regular singularities. Let
$\Phi_{f}\mathcal{M}$ denote the nearby cycle functor and let $i_{t}$ denote
the embedding of $f^{-1}(t)$ into $Z$. Then
\[
SS(\Phi_{f}(\mathcal{M}))=\lim_{t\rightarrow0}SS((i_{t})_{\ast}(i_{t})^{\ast
}\mathcal{M}).
\]

\end{thm}

\begin{proof}
[Proof of Lemma \ref{lem:AVJ}]From Theorem \ref{thm:CVnearby} we obtain
\[
SS(\Phi_{f}(\mathcal{M}))=\lim_{t\rightarrow0}\rho^{\vee}(t)SS(\mathcal{M})
\]
and passing to characteristic varieties we get%
\begin{equation}
CV(\Phi_{f}(\mathcal{M}))=\lim_{t\rightarrow0}\rho^{\vee}(t)CV(\mathcal{M})
\label{eq:CVGJ}%
\end{equation}

Identify $\mathfrak{n}^{\ast}$ with the subspace of ${\mathfrak{g}}^{\ast}$
consisting of vectors having negative weights under the action of $d\rho
^{\vee}(1)$, $[\mathfrak{n},\mathfrak{n}]^{\bot}$ with vectors having weights
at least $-1$ and $\Psi$ with those having weight $-1$. Then by
\eqref{eq:CVAV} and \eqref{eq:AVng}, $\mathrm{As}\mathcal{V}_{\mathfrak{n}%
}(M)\cap\Psi$ is obtained by intersecting $CV(\mathcal{M})$ with the constant
bundle $X\times\lbrack\mathfrak{n},\mathfrak{n}]^{\bot}$, projecting to the
second coordinate and then further projecting to $\Psi$. Denote this operation
on subvarieties of $T^{\ast}X$ by $p_{\Psi}$. Since the characteristic variety
is a conical set (in cotangent directions), $p_{\Psi}(\rho^{\vee
}(t)CV(\mathcal{M}))$ does not depend on $t$. Since $X$ is complete we get
\[
p_{\Psi}(\lim_{t\rightarrow0}\rho^{\vee}(t)CV(\mathcal{M}))\supset p_{\Psi
}(CV(\mathcal{M}))
\]
Thus we get
\[
p_{\Psi}(CV(\Phi_{f}(\mathcal{M})))\supset p_{\Psi}(CV(\mathcal{M}))
\]
Lemma \ref{lem:AVJ} follows now from Theorem \ref{thm:GeoJac}.
\end{proof}

\subsubsection{Counterexamples to stronger statements}

\label{subsubsec:CEx}

First of all, Lemma \ref{lem:AVJ} does not generalize to arbitrary
finitely-generated $\mathfrak{n}$-modules. Indeed, let $G=GL(3,{\mathbb{R}})$
and consider the identification of $\mathfrak{n}$ with the Heisenberg Lie
algebra $\left\langle x,\frac{d}{dx},1\right\rangle $ acting on $V=\mathbb{C}%
\left[  x\right]  $. Then $C(V)$  vanishes.

Next one might ask whether for a Harish-Chandra module $M$ the inclusion in
Lemma \ref{lem:AVJ} holds without the intersection with $\Psi$, i.e.
$\mathrm{As}\mathcal{V}_{\mathfrak{n}}(M)\subset\mathrm{As}\mathcal{V}(JM)$.
The answer is no, as shown by the following example.

Let $G=GL(3,\mathbb{R})$ and let $\mathfrak{g}$ be its complexified Lie
algebra. Let $\mathfrak{b}$ be the Borel subalgebra of upper-triangular
matrices, let $\mathfrak{n}$ be its nilradical, and let $\mathfrak{s}$ be the
space of symmetric matrices. Using the trace form, we identify $\mathfrak{g}$
with $\mathfrak{g}^{\ast}$ and $\mathfrak{\bar{n}}$ with $\mathfrak{n}^{\ast}%
$. Let $M$ be a degenerate principal series representation corresponding to
the $\left(  2,1\right)  $ parabolic. Then we have
\[
\mathrm{AnV}_{\mathfrak{g}}(M)=\mathcal{R},\quad\mathrm{AsV}_{\mathfrak{g}%
}(M)=\mathcal{R}\cap\mathfrak{s}%
\]
where $\mathcal{R}$ is the set of nilpotent matrices of rank $\leq1$.

For a lower triangular matrix let $a,b,c$ denote its entries as shown
\[
\left[
\begin{array}
[c]{ccc}%
0 & 0 & 0\\
a & 0 & 0\\
b & c & 0
\end{array}
\right]
\]
Then we get
\begin{align*}
\mathrm{AsV}_{\mathfrak{n}}(M)  &  =pr_{\mathfrak{\bar{n}}}\left(
\mathcal{R}\cap\mathfrak{s}\right)  =\left\{  a^{2}b^{2}+a^{2}c^{2}+b^{2}%
c^{2}=0\right\} \\
\mathrm{AsV}_{\mathfrak{n}}(JM)  &  \subset\mathcal{R}\cap\mathfrak{\bar{n}%
}=\left\{  ac=0\right\}
\end{align*}

\section{Proof of Theorems \ref{thm:OurMat}, \ref{thm:determine}}

\label{sec:PfOurMat}

We start with two preliminary subsections.

\subsection{Nilpotent orbits and wavefront sets}

\label{subsec:PrelComp}

%?? Dima edited this section.

Let $G$ be a real reductive group.
Let
$\mathcal{N\subset}$ ${\mathfrak{g}}^{\ast}$ denote the null cone, with
$\mathcal{N}_{\theta}=\mathcal{N\cap}{\mathfrak{k}}^{\perp}$ and
$\mathcal{N}_{0}=\mathcal{N\cap}{\mathfrak{g}}_{0}^{\ast}$ as before. The
groups $G_{\mathbb{C}},K_{\mathbb{C}}$ and $G$ act with finitely many orbits
on $\mathcal{N}$, $\mathcal{N}_{\theta}$ and $\mathcal{N}_{0}$ respectively.
We write ${\mathcal{O}}^{\prime}\leq{\mathcal{O}}$ if ${\mathcal{O}}^{\prime}$
is contained in the closure $\overline{{\mathcal{O}}}$ of ${\mathcal{O}}$, and
we refer to $\leq$ as the closure order.

The  Kostant-Sekiguchi correspondence (\cite{Sek}) provides  a bijection between $G$-orbits on
$\mathcal{N}_{0}$ and $K_{\mathbb{C}}$-orbits on $\mathcal{N}_{\theta}.$ Let us briefly recall its construction. Let $\cO\subset \cN_0$ be a nilpotent $G$-orbit. Then one can choose an element $X \in \cO$ and an $sl(2)$-triple $(H,X,Y)$  satisfying the  \emph{Cayley} property
$$\theta(H)=-H, \, \theta(X)=-Y, \, \theta(Y)=-X.$$
%
% Let $X\in \cO$.
% By the Jacobson-Morozov theorem, every nilpotent element $X \in \cN_0$ can be completed to an $sl(2)$-triple $(H,X,Y)$ in $\fg_0$. Moreover, this triple can be $G$-conjugated to a \emph{Cayley} triple, \textit{i.e.} one satisfying
% ?
The Kostant-Sekiguchi correspondence attaches to $\cO$  the $K_{\C}$-orbit $KS(\cO)$ of $X'=\frac{1}{2}(X+Y+iH)\in \cN_{\theta}.$
%\end{defn}
% The
% bijection, which we denote by ${\mathcal{O}}\mapsto\operatorname{KS}%
% ({\mathcal{O}})=\operatorname{KS}_{G}({\mathcal{O}})$, is called the
% Kostant-Sekiguchi correspondence.

\begin{thm}\label{thm:KS}  %For a nilpotent $G$-orbit $\cO\subset \cN_0$ we denote by $KS(\cO)$ %the corresponding orbit
\begin{enumerate}[(a)]
\item \label{it:KSDef} The map $\cO \mapsto KS(\cO) $ gives a well-defined bijection between $G$-orbits on
$\mathcal{N}_{0}$ and $K_{\mathbb{C}}$-orbits on $\mathcal{N}_{\theta}$.
\item \label{it:KSOrb} The orbits ${\mathcal{O}}$ and
$KS({\mathcal{O}})$ lie in the same complex coadjoint orbit.

\item \label{it:KSFG}For $F_G$ as in formula \eqref{=FG} we have
$
F_{G}\cdot\operatorname{KS}({\mathcal{O}})=\operatorname{KS}(F_{G}%
\cdot{\mathcal{O}})$

\item \label{it:KSLevi} If  $L\subset G$ is a
standard Levi subgroup, then we have $\operatorname{KS}_{L}({\mathcal{O}}\cap\mathfrak{l}^{\ast
})\subset\operatorname{KS}_{G}({\mathcal{O}}).$

\item \label{it:KSOrder} The correspondence $KS$ preserves
the closure order on nilpotent orbits.

\end{enumerate}
\end{thm}
\begin{proof}
Part \eqref{it:KSDef} is the main result of  \cite{Sek}.
For part \eqref{it:KSOrb} note that $X'$ is the Cayley transform of $X$, obtained via conjugation  by $\exp(-\frac{\pi i}{4}(X+Y))$. For part \eqref{it:KSFG}  note that the action of $F_G$ commutes with $\theta$ and with the complex conjugation, thus it maps Cayley triples to Cayley triples and commutes with the map $KS$. For part \eqref{it:KSLevi} note that $L$ is $\theta$-stable and $\theta|_{L}$ is a Cartan involution of $L$ and hence a Cayley triple  in $\fl_0$ is a Cayley triple in $\fg_0$.
Part \eqref{it:KSOrder} is the main result of \cite{BS}.
\end{proof}

% The orbits ${\mathcal{O}}$ and
% $KS({\mathcal{O}})$ lie in the same complex coadjoint orbit.
%
%
% \begin{theorem}
% [\cite{BS}]\label{thm:KS} The Kostant-Sekiguchi correspondence $KS$ preserves
% the closure order.
% \end{theorem}
%
%
% Let $F_{G}$ be as (\ref{=FG}) then $F_{G}$ acts on both $\mathcal{N}_{\theta}$
% and $\mathcal{N}_{0}$, and we have
%
% \begin{proposition}
% \label{prop:KSL}
% %?? reference ? maybe it follows from the fact that  $\cO$ and $KS(\cO)$ %lie in the same complex coadjoint orbit? ?? Is $L$  $\theta$-stable when %we apply this?
% If ${\mathcal{O}}$ is a $G$-orbit on $\mathcal{N}_{0}$ and $L\subset G$ is a
% standard Levi subgroup, then we have%
% \[
% F_{G}\cdot\operatorname{KS}({\mathcal{O}})=\operatorname{KS}(F_{G}%
% \cdot{\mathcal{O}}),\operatorname{KS}_{L}({\mathcal{O}}\cap\mathfrak{l}^{\ast
% })\subset\operatorname{KS}_{G}({\mathcal{O}}).
% \]
%
% \end{proposition}
% \begin{proof}
%
% \end{proof}

In addition to the associated variety,
there is a further invariant of $\pi$ called the $\emph{wavefront\ set}$,
which was defined in \cite{HowWF} in terms of the global character of $\pi$.
This is a $G$-invariant set
\[
\mathrm{WF}\left(  \pi\right)  \subset i\mathcal{N}_{0}%
\]
which, by \cite{Ros1,Ros2}, coincides with the \emph{asymptotic support} of
$\pi$ introduced in \cite{BV-Unip}. As conjectured in \cite{BV-Unip} and
proved in \cite{SV} one also has

\begin{theorem}
\label{thm:SV} If $\left(  \pi,W\right)  \in\mathcal{M}(G)$ then
$\mathrm{WF}(\pi)=i\operatorname{KS}\left(  \mathrm{As}\mathcal{V}\left(
\pi^{HC}\right)  \right)  .$
\end{theorem}

Using Theorem \ref{thm:Vog} it follows that for all $\pi\in\mathcal{M}(G)$ we have%
\begin{equation}
\mathrm{An}\mathcal{V}\left(  \pi\right)  =G_{\mathbb{C}}\cdot\mathrm{As}%
\mathcal{V}\left(  \pi\right)  =G_{\mathbb{C}}\cdot\mathrm{WF}\left(
\pi\right)  \label{=GCorb}%
\end{equation}

We will also need the following result.
\begin{thm}
[{\cite[Theorem A]{MatComp}}]\label{thm:MatComp} For $\pi\in\mathcal{M}(G),$
we have
\[
\operatorname{WF}(\pi)\cap\Psi_{0}^{\times}=\Psi_{0}^{\times}(\pi).
\]

\end{thm}

Now suppose that $G$ is a complex reductive group, regarded as a real group.
Then the real Lie algebra ${\mathfrak{g}}_{0}$ is already a complex Lie
algebra, and we have ${\mathfrak{g}}\cong{\mathfrak{g}}_{0}\times
{\mathfrak{g}}_{0}$, and ${\mathfrak{g}}_{0}$ is diagonally embedded into
${\mathfrak{g}}$. The Lie algebra ${\mathfrak{k}}$ is also isomorphic to
${\mathfrak{g}}_{0}$, and is embedded into ${\mathfrak{g}}$ by $X\mapsto
(X,\theta\left(  X\right)  )$. For a nilpotent orbit ${\mathcal{O}}%
\subset\mathcal{N}({\mathfrak{g}})$ we have ${\mathcal{O}}={\mathcal{O}}%
_{1}\times{\mathcal{O}}_{2}$ where ${\mathcal{O}}_{i}\subset\mathcal{N}%
({\mathfrak{g}}_{0})$. However, if ${\mathcal{O}}$ intersects $i{\mathfrak{g}%
}_{0}^{*}\subset{\mathfrak{g}^{*}}$ or ${\mathfrak{k}}^{\bot}\subset
{\mathfrak{g}^{*}}$ then ${\mathcal{O}}_{1}={\mathcal{O}}_{2}$, and thus
${\mathcal{O}}$ is defined by a single nilpotent orbit in ${\mathfrak{g}}_{0}%
$. By Theorem \ref{thm:Vog}, only orbits intersecting ${\mathfrak{k}%
}^{\bot}$ can be open orbits in the annihilator variety of an admissible
representation $\pi\in\mathcal{M}(G)$, and thus we will be only interested in
such orbits.

\subsection{The Jacquet restriction functor}\label{subsec:JacRes}

As before let $B$ be the fixed Borel subgroup of $G$. Let $P\supset B$ be a
standard parabolic subgroup, fix a Levi decomposition $P=LU$ and let
$\mathfrak{u}$ be the complexified Lie algebra of $U$. For $\left(
\pi,W\right)  \in\mathcal{M}(G)$ we have a natural representation of $L$ on
the space $W/\overline{\mathfrak{u}W}$ where $\overline{\mathfrak{u}W}$
denotes the closure of $\mathfrak{u}W$ in $W$ (by an unpublished result of Casselman, $\mathfrak{u}W$
is already closed in $W$, but we will not need this fact).
This representation is usually
denoted by $r_{P}\left(  \pi\right)  $ and referred to as the Jacquet
restriction functor, for simplicity we will write $r_{P}\left(  \pi\right)
=\pi_{P}$. Its main properties are summarized below.

\begin{thm}
\label{thm:Jac}
%Let $P=MU$ be a parabolic subgroup of $G$. Let $\mathfrak{u}:=Lie(N)$.
%Let $K$ be the maximal compact subgroup
%of $G$ and let $M_{K}=M \cap K$ be the maximal compact subgroup of $M$. Then
\cite[3.8.2 and 5.2.3]{Wal1}

\begin{enumerate}[(a)]
\item $\pi_{P}\in\mathcal{M}(L)$

\item $r_{P}$ is left adjoint to the parabolic induction functor
$\mathcal{M}(L)\rightarrow\mathcal{M}(G)$.

\item $(\pi_{P})^{HC}=C_{\mathfrak{u}}(\pi^{HC}):=\pi^{HC}/\mathfrak{u}%
\pi^{HC}$.
\end{enumerate}
\end{thm}

Recalling the definition of $\Psi=\left(  \mathfrak{n/n}^{\prime}\right)
^{\ast}$ etc. from (\ref{=v}) we write $\Psi_{G}$ to denote its dependence on
$G$. For each standard parabolic $P=LU$ we can regard $\Psi_{L}$ as a subset
of $\Psi_{G}$ as follows: $\Psi_{L}\approx\left\{  \psi\in\Psi_{G}%
:\psi|_{\mathfrak{u}}=0\right\}  $. It follows immediately that for $\pi
\in\mathcal{M}(G)$ we have%
\[
\Psi\left(  \pi_{P}\right)  =\Psi\left(  \pi\right)  \cap\Psi_{L}%
\]

% Since $F_{G}$ normalizes $G$, we get an action $\pi\mapsto
% \pi^{a}$ on representations of $G$, and we define
% \[
% \widetilde{\pi}=\bigoplus\left\{  \pi^{a}:a\in F_{G}\right\}
% \]
% We emphasize that the set symbol means that we only sum over \emph{distinct}
% transforms $\pi^{a}$ of $\pi.$ Now $F_{G}$ normalizes ${\mathfrak{g}}_{0}$ and
% ${\mathfrak{n}}_{0}$, hence it acts on $\Psi_{0}$ and we have%
% \[
% \Psi\left(  \pi^{a}\right)  =a\cdot\Psi(\pi)\text{, }\Psi\left(
% \widetilde{\pi}\right)  =F_{G}\cdot\Psi(\pi)=\bigcup\left\{  a\cdot\Psi
% (\pi):a\in F_{G}\right\}
% \]

%?? Dima
Next, recall the definition of $F_{G}\subset Int(\g_{\mathbb{C}})$ from formula \eqref{=FG} and denote $\widetilde {\Psi_0}(\pi):=F_G \cdot \Psi_0(\pi)$.

\begin{lem}\label{lem:FGIso}
Let $\Pi_0(G)$ denote the set of simple restricted roots. Then we have a natural isomorphism $$a \mapsto \eps_a:F_G \simeq (\Z/2\Z)^{\Pi_0(G)},$$
where for $\beta\in \Pi_0(G)$, $\eps_a(\beta)$ is the sign $\pm 1$ by which  $a\in F_G$ acts on $\fg_{\beta}$.\end{lem}
\begin{proof}
For $x \in \fa_0$, the operator $a=ad(\exp(ix))$ acts on each restricted root space $\fg_{\beta}$ by the scalar $\exp(i\beta(x))$. Thus $a \in F_G$ if and only if $x\in \pi \Lambda$, where $\Lambda$ is the (restricted) coroot lattice. Then  $x \mapsto ad(\exp(i\pi x))$ defines a surjection $\Lambda \onto F_G$ with kernel equal to $2\Lambda$. Thus we get $F_G \simeq \Lambda/2\Lambda \simeq (\Z/2\Z)^{\Pi_0(G)}$.
\end{proof}

\begin{cor}\label{cor:FGsur}
Let $L \subset G$ be a standard Levi subgroup. Then $\fl$ is $F_G$-invariant and the restriction map $a \mapsto a|_{\fl}$ gives a surjection $F_G \onto F_L$.
\end{cor}
\begin{proof}
Since $L$ is a standard Levi subgroup, we have ${\mathfrak{a}}_{0}\subset{\mathfrak{l}_0}$, and hence $\fl$ is invariant under $F_G\subset ad(\exp(i\fa_0))$.
By Lemma \ref{lem:FGIso} we have $F_G \simeq (\Z/2\Z)^{\Pi_0(G)}$ and $F_L \simeq (\Z/2\Z)^{\Pi_0(L)}$, and the map $a \mapsto a_{\fl}$ corresponds to the restriction from $\Pi_0(G)$ to $\Pi_0(L)\subset \Pi_0(G)$, and therefore it is a surjection.
\end{proof}

\begin{lemma}
\label{lem:WFLevi} For $\pi\in\mathcal{M}(G)$ and $P=LU\supset B$ we have%
%\begin{gather*}
$$ (a) \quad
\mathrm{As}\mathcal{V}\left(  \pi_{P}\right)  \subset\mathrm{As}%
\mathcal{V}\left(  \pi\right)  \cap\mathfrak{l}^{\ast}, \quad (b) \quad \mathrm{WF}\left(
\pi_{P}\right)  \subset\mathrm{WF}\left(  \pi\right)  \cap\mathfrak{l}^{\ast
}$$
\end{lemma}

\begin{proof}
Part (a)  follows from Lemma \ref{lem:Quot}, and this implies part (b)
by Theorems \ref{thm:KS} and \ref{thm:SV} \end{proof}

\begin{lemma}
\label{lem:PsiLevi} For $\pi\in\mathcal{M}(G)$ we have
$$
(a) \quad \Psi_{0}\left(  \pi\right)  =\bigcup_{P\supset B}\Psi_{0}^{\times}\left(  \pi
_{P}\right)  , \quad (b) \quad \Psi_{0}\left(  \pi^{HC}\right)  =\bigcup_{P\supset B}\Psi_{0}^{\times
}\left(  \pi_{P}^{HC}\right)  ,\quad (c) \quad  \widetilde{\Psi_{0}}\left(  {\pi}\right)
=\bigcup_{P\supset B}\widetilde{\Psi_{0}^{\times}}\left(  {\pi_{P}}\right)
$$

\end{lemma}

\begin{proof}
Parts (a) and (b) follow from the definition of
$\pi_{P}$ and Lemma \ref{lem:non-deg-par}. For part (c) we note that by Corollary \ref{cor:FGsur} we have
% Since ${\mathfrak{a}}_{0}\subset{\mathfrak{l}}$ we have a surjection ?? $a\mapsto
% a_{P}:F_{G}\rightarrow F_{L}$, such that $\left(  a\cdot\pi\right)  _{P}%
%=a_{P}\cdot\pi_{P}.$ Thus we get%
\[
{\widetilde\Psi_{0}}\left(  {\pi}\right)  =\bigcup_{a\in F_{G}}a\cdot{\Psi}%
_{0}\left(  \pi\right)  =\bigcup_{a\in F_G}\bigcup_{P}a\cdot{\Psi_{0}}^{\times}\left( \pi_{P}\right)  =\bigcup_{P}\bigcup_{a\in F_{L_P}}a\cdot{\Psi_{0}}^{\times}\left(
\pi_{P}\right)  =\bigcup_{P}{\widetilde\Psi_{0}}^{\times}\left(
{\pi_{P}}\right),
\]
where $L_P$ denotes the Levi subgroup of $P$.
\end{proof}

\begin{lemma}
\label{lem:KS} If $i\lambda\in\Psi_{0}$ then $\lambda\in pr_{\mathfrak{n}%
^{\ast}}(\overline{\mathrm{KS}(G\cdot\lambda)})$.
\end{lemma}

\begin{proof}
By Corollary \ref{cor:PrOnto} $pr_{\mathfrak{n}^{\ast}}$ projects any
irreducible component of $\mathcal{N}_{\theta}$ onto $\mathfrak{n}^{\ast}$.
This implies the result if $\lambda$ is principal nilpotent. For general
$\lambda$, we choose a standard parabolic $LU$ such that $\lambda$ vanishes on
$\mathfrak{u}$ and is principal nilpotent on $\mathfrak{l\cap n}$. The result
now follows from Theorem \ref{thm:KS}.
\end{proof}

%?? This lemma also follows from the explicit formula for KS.

\subsection{Proofs of the theorems}

We first prove Theorem \ref{thm:OurMat}.

\begin{proof}
[Proof of Theorem \ref{thm:OurMat}]By Theorems \ref{thm:KosCHM} and
\ref{thm:MatComp}, for all $P$ we have
\[
\Psi_{0}^{\times}(\pi_{P})\subset\mathrm{WF}\left(  \pi_{P}\right)  \cap
\Psi,\quad\Psi_{0}^{\times}\left(  \pi_{P}^{HC}\right)  =\widetilde\Psi_{0}^{\times
}\left(  \pi_{P}\right)  .
\]
Taking the union over all $P\supset B$ and using Lemmas \ref{lem:WFLevi} and \ref{lem:PsiLevi} we
get
\begin{equation}
\Psi_{0}(\pi)\subset\mathrm{WF}\left(  \pi\right)  \cap\Psi,\quad\Psi_{0}%
(\pi^{HC})=\widetilde\Psi_{0}\left(  {\pi}\right ) =F_G \cdot \Psi_0(\pi) . \label{=p1}%
\end{equation}

By Lemma \ref{lem:KS}, Theorem \ref{thm:SV}, and Theorem \ref{thm:Main}, we
get%
\[
\mathrm{WF}\left(  \pi\right)  \cap\Psi\subset pr_{\mathfrak{n}^{\ast}}\left(
\mathrm{KS}(\mathrm{WF}\pi)\right)  =pr_{\mathfrak{n}^{\ast}}\left(
\mathrm{As}\mathcal{V}(\pi^{HC})\right)  =\Psi(\pi^{HC})
\]
Since $\mathrm{WF}\left(  \pi\right)  \subset\Psi_{0}$ , it follows that
\begin{equation}
\mathrm{WF}\left(  \pi\right)  \cap\Psi\subset\Psi_{0}\left(  \pi^{HC}\right)
\label{=p2}%
\end{equation}

Combining (\ref{=p1}) and (\ref{=p2}) we obtain (\ref{=ourmat1}).

%Next we show that $\Psi(\pi^{HC})$ is the Zariski closure $\overline{\Psi
%(\pi)}$. For generic $\pi$, we have $\Psi(\pi^{HC})=\Psi$, while $\Psi(\pi)$
%is open in $\Psi_{0}$ by Proposition \ref{prop:nondeg}, and the result
%follows. In any case, by Theorem \ref{thm:Main}, $\Psi(\pi^{HC})$ is the image
%of a closed set under a finite map (see Corollary \ref{cor:PrFinite}) and thus
%it is Zariski closed. Also by Lemma \ref{lem:WFLevi} we have
%\[
%\Psi\left(  \pi\right)  =\bigcup_{P}\Psi^{\times}\left(  \pi_{P}\right)
%,\Psi\left(  \pi^{HC}\right)  =\bigcup_{P}\Psi^{\times}\left(  \pi_{P}%
%^{HC}\right)
%\]
%Thus it suffices to prove the result for all $\pi_{P}$ that are generic, but
%this follows by the earlier argument.

Finally if $G$ is a complex group or if $G=GL\left(  n,\mathbb{R}\right)  $
then each complex nilpotent orbit contains at most one real orbit. This has
two consequences. First by (\ref{=GCorb}) it follows that
\[
\mathrm{WF}\left(  \pi\right)  \cap\Psi=\mathrm{An}\mathcal{V}\left(
\pi\right)  \cap\Psi_{0}%
\]
Second, since the group $F_{G}$ permutes the real forms of a complex nilpotent
orbit, it acts trivially on orbits and we get $\widetilde\Psi({\pi})=\Psi
(\pi)$. Thus (\ref{=ourmat2}) follows from (\ref{=ourmat1}).
\end{proof}

We next prove Theorem \ref{thm:determine} using Theorem \ref{thm:ClassGroup}.

\begin{proof}
[Proof of Theorem \ref{thm:determine}]In view of (\ref{=GCorb}) it suffices to
show that $\mathrm{WF}(\pi)$ is determined by $\Psi_{0}\left(  \pi\right)
=\mathrm{WF}(\pi)\cap\Psi_{0}$.
For part \eqref{it:detGL} of Theorem \ref{thm:determine} this is straightforward since every $G$-orbit in
$\mathrm{WF}(\pi)$ intersects $\Psi_{0}$. For part \eqref{it:detC}, note that since $\pi$ is
irreducible then by \cite{JosVar} there is a complex nilpotent orbit
${\mathcal{O}}$ such that $\mathrm{An}\mathcal{V}\left(  \pi\right)
=\overline{{\mathcal{O}}}$. Since $G$ is itself a complex group,
${\mathcal{O}}_{0}={\mathcal{O}}\cap{\mathfrak{g}}_{0}^{\ast}$ is a single
$G$-orbit and we have $WF\left(  \pi\right)  =i\overline{{\mathcal{O}}_{0}}$.
Thus it suffices to show that ${\mathcal{O}}_{0}$ is determined by
$\overline{{\mathcal{O}}_{0}}$ $\cap\Psi_{0}$, which follows from Theorem
\ref{thm:ClassGroup}.
\end{proof}

\subsection{Some remarks on Theorem \ref{thm:OurMat}}

\label{sec:RemOurMat}

The action of $F_{G}$ is not very significant in Theorem \ref{thm:OurMat}. For
instance, let $\psi\in\mathrm{WF}(\pi)\cap\Psi_{0}$ and choose a parabolic
subgroup $P=LU$ such that $\psi$ is a principal nilpotent element in
$\mathfrak{l}^{\ast}$. Then we have shown that $\mathrm{WF}(\pi_{P})$ contains
some (real) principal nilpotent orbit. The action of $F_{L}\subset F_{G}$ is
used to permute the real principal nilpotent orbits of $L$, but if $G$ is
classical then there are only 2 such orbits (since $L$ is then a product of a
classical group with $GL_{n_{i}}$).

We next give an example to show that $\widetilde{\Psi}_{0}\left(  {\pi}\right)
=\Psi_{0}(\pi^{HC})$ can be a \emph{proper} subset of $\mathrm{An}%
\mathcal{V}(\pi)\cap\Psi_{0}$. Let $P\approx GL\left(  n,{\mathbb{C}}\right)
\ltimes Herm_{n}$ be the Siegel-Shilov parabolic subgroup of $U\left(
n,n\right)  $ where $Herm_{n}$ is the space of $n\times n$ Hermitian matrices.
Let $\pi$ be the corresponding unitary degenerate principal series
representation considered by Kashiwara and Vergne in \cite{KVShi,KVWF}. As
shown in \cite{KVShi} $\pi$ decomposes into $n+1$ constituents $\pi_{0}%
,\ldots,\pi_{n}$, all of the same Gelfand-Kirillov dimension ($\pi_{i}$ is
denoted $\pi_{n-i,i}$ in \cite{KVShi}). On the other hand, the complex
Richardson orbit ${\mathcal{O}}_{{\mathbb{C}}}$ for $P$ contains $n+1$ real
orbits ${\mathcal{O}}_{0},{\mathcal{\ldots}},{\mathcal{O}}_{n}$ as well, and
by a result of Barbasch (see \cite{MatTra}) the associated wavefront cycle of
$\pi$ is $\sum\left[  \overline{{\mathcal{O}}_{i}}\right]  $, i.e. all
multiplicities are $1.$ It follows that the wavefront cycle of each $\pi_{i}$
is the closure of a \emph{single} real orbit, which after relabeling we may
assume to be $\overline{{\mathcal{O}}_{i}}$.

For the group $U(2,2)$, ${\mathcal{O}}_{{\mathbb{C}}}$ consists of matrices
with rank 2 and square 0. It contains three real nilpotent orbits
${\mathcal{O}}_{0},{\mathcal{O}}_{1},{\mathcal{O}}_{2}$ whose representatives
are the respective block matrices
\[
\left[
\begin{array}
[c]{cc}%
0 & I_{2}\\
0 & 0
\end{array}
\right]  \quad\left[
\begin{array}
[c]{cc}%
0 & \operatorname{diag}(1,-1)\\
0 & 0
\end{array}
\right]  ,\quad\left[
\begin{array}
[c]{cc}%
0 & -I_{2}\\
0 & 0
\end{array}
\right]
\]
The group $F_{G}$ preserves ${\mathcal{O}}_{1}$, and permutes ${\mathcal{O}%
}_{0}$ with ${\mathcal{O}}_{2}$. It is easy to see that ${\mathcal{O}}_{1}$
intersects $\Psi$, while ${\mathcal{O}}_{0}$ and ${\mathcal{O}}_{2}$ do not.
Thus we get that
\[
\Psi(\pi_{0}^{HC})\cap\Psi_{0}=\Psi_{0}(\pi_{0}^{HC})=\left(  \overline
{{\mathcal{O}}_{0}}\cup\overline{{\mathcal{O}}_{2}}\right)  \cap\Psi_{0},
\]
which is not equal to $\overline{{\mathcal{O}}_{{\mathbb{C}}}}\cap\Psi
_{0}=\mathrm{An}\mathcal{V}(\pi_{0})\cap\Psi_{0}$. Hence $\Psi_{0}(\pi
_{0}^{HC})$ is a proper subset of $\mathrm{An}\mathcal{V}(\pi_{0})\cap\Psi
_{0}$.

Theorem \ref{thm:OurMat} determines only $\Psi_{0}(\pi^{HC})$, not $\Psi
(\pi^{HC})$. From Lemma \ref{lem:AVVSupp} we see that $\Psi(\pi^{HC})$ is
Zariski closed, so one might ask whether $\Psi(\pi^{HC})$ is the Zariski
closure of $\Psi_{0}(\pi^{HC})$. Using the arguments of the above two
subsections it can be easily proven for split groups. However, this statement
does not generalize to all quasi-split groups. The representation $\pi_{1}$
provides a counter-example for $U(2,2)$, and  degenerate principle series
representations (\textit{i.e.} sections of a line bundle on the projective space
$\mathbb{CP}^{2}$) provide a counter-example for $\mathfrak{sl}_{3}%
({\mathbb{C}})$. This can be shown using Theorem \ref{thm:Main}.

%\subsection{Some remarks on Theorem \ref{thm:determine}}

\subsection{Wave-front set of the Jacquet restriction}

\label{subsec:WFJac}

%??Dima

Theorem \ref{thm:OurMat} implies the following proposition.

\begin{prop}
\label{prop:WFLevi} Let $\pi\in\mathcal{M}(G)$ and $P=LU$ be a standard
parabolic subgroup. Suppose that every maximal orbit of $L$ in $\mathrm{WF}%
(\pi)\cap\mathfrak{l}^{\ast}$ intersects $\Psi$. Then
\[
\mathrm{WF}(\pi_{P})\subset\mathrm{WF}(\pi)\cap\mathfrak{l}^{\ast} \subset
F_{L} \cdot\mathrm{WF}(\pi_{P})
\]

\end{prop}

\begin{proof}
The first containment is part of Lemma \ref{lem:WFLevi}. For the second one we
get from Theorem \ref{thm:OurMat}
\[
\mathrm{WF}(\pi)\cap\mathfrak{l}^{\ast} \cap\Psi\subset\Psi_{0}(\pi^{HC}%
)\cap\mathfrak{l}^{\ast}\subset\Psi_{0}(\pi_{P}^{HC})\subset F_{L}
\cdot\mathrm{WF}(\pi_{P})
\]
Since $\mathrm{WF}(\pi_{P})$ and $\mathrm{WF}(\pi)\cap\mathfrak{l}^{\ast}$ are
$L$-stable, and every maximal orbit of $L$ in $\mathrm{WF}(\pi)\cap
\mathfrak{l}^{\ast}$ intersects $\Psi$, we conclude that $\mathrm{WF}(\pi
)\cap\mathfrak{l}^{\ast} \subset F_{L} \cdot\mathrm{WF}(\pi_{P})$.
\end{proof}

\begin{cor}
\label{cor:WFLevi} Let $\pi\in\mathcal{M}(G)$ and suppose $P=LU$ is a standard
parabolic subgroup such that $L$ is a product of several $GL_{n_{i}}$ factors.
Then
\[
\mathrm{WF}(\pi_{P})=\mathrm{WF}(\pi)\cap\mathfrak{l}^{\ast}\text{ for all
}\pi\in\mathcal{M}\left(  G\right)
\]

\end{cor}

%?? Should we prove it?
It is very interesting for us to know whether $WF(\pi_{P})=WF(\pi
)\cap\mathfrak{l}^{*}$ without any assumption on $L$. Obviously, Proposition
\ref{prop:WFLevi} and Corollary \ref{cor:WFLevi} are false for  the Jacquet functor for p-adic groups, e.g. for cuspidal $\pi$.

%??Dima

\section{Applications to $GL\left(  n\right)  $\label{sec:GL}}

Let $G_{n}=GL(n,F)$ with $F=\mathbb{R}$ or $\mathbb{C}$, and suppose $\pi
\in\mathcal{M}(G_{n})$. \cite{AGS} gives several definitions for the
derivative of $\pi$, inspired by the p-adic notion defined in
\cite{BZ-Induced}. Here we will use the following definition. Let
$P_{n}\subset G_{n}$ be the mirabolic subgroup, consisting of matrices with
last row $(0,\dots,0,1)$, then $P_{n}\approx G_{n-1}\ltimes V_{n}$ where
$V_{n}\approx F^{n-1}\subset P_{n}$ is imbedded as the last column. If
$\left(  \tau,V\right)  $ is a representation of ${\mathfrak{p}}_{n}$ and
$\xi$ is a character of $\mathfrak{v}_{n}$ we can consider the coinvariants%
\[
C_{\xi}\left(  \tau\right)  =V/{\operatorname{Span}}\{\tau\left(  {X}\right)
v-\xi({X})v:v\in V,\,{X}\in\mathfrak{v}_{n}\}.
\]

Let $\xi_{0}$ be the trivial character of $\mathfrak{v}_{n}$ and let $\xi_{1}$
be the character given by
\[
\xi_{1}(x_{1},\dots,x_{n-1}):=\sqrt{-1}\operatorname{Re}x_{n-1}.
\]
The normalizers of $\xi_{0}$ and $\xi_{1}$ in $G_{n-1}$ are $G_{n-1}$ and
$P_{n-1}$ respectively, and hence $C_{\xi_{0}}\left(  \tau\right)  $ and
$C_{\xi_{1}}\left(  \tau\right)  $ are representations of $\mathfrak{g}_{n-1}$
and ${\mathfrak{p}}_{n-1}$, respectively. We write $\Phi(\tau)=|\det
|^{-1/2}\otimes C_{\xi_{1}}\left(  \tau\right)  $ and we define the $k$-th
derivative of $\tau$ to be the following representation of $\mathfrak{g}%
_{n-k}$%
\[
B^{k}(\tau)=C_{\xi_{0}}\Phi^{k-1}(\tau).
\]
By \cite[Proposition 3.0.3]{AGS} if $\sigma\in\mathcal{HC}(G_{n})$ then
$B^{k}(\sigma)=B^{k}(\sigma|_{{\mathfrak{p}}_{n}})\in\mathcal{HC}(G_{n-k})$,
i.e. $B^{k}(\sigma)$ is admissible.

Theorem \ref{thm:determine} allows one to calculate the annihilator variety
$\mathrm{An}\mathcal{V}\left(  B^{k}(\sigma)\right)  $ in terms of
$\mathrm{An}\mathcal{V}\left(  \sigma\right)  $. For simplicity we consider
the case of $G_{n}=GL(n,\mathbb{R})$, since the case of $GL(n,\mathbb{C})$ is
very similar. Note that $\mathrm{An}\mathcal{V}\left(  \sigma\right)  $ is a
union of complex nilpotent orbits $\mathcal{O}_{\lambda}\subset\mathfrak{g}%
_{n}^{\ast}=\mathfrak{gl}\left(  n,\mathbb{C}\right)  ^{\ast}$, which are
indexed by partitions $\lambda$ of $n$ as in section \ref{subsec:PrelComp}.
Also the nilradical $\mathfrak{n}$ consists of upper triangular matrices and
for each partition $\lambda$ we consider the character%

\[
\psi_{\lambda}\left(  X\right)  =\sqrt{-1}\left(  \sum\nolimits_{j\notin
S_{\lambda}}X_{j,j+1}\right)
\]
where $S_{\lambda}$ is the index set of partial sums $\left\{  \lambda
_{k}+\cdots+\lambda_{l}:1\leq k<l\right\}  $ with $l=\mathrm{length}\left(
\lambda\right)  $; then we have $\psi_{\lambda}\in\mathcal{O}_{\lambda}$.

\begin{lemma}
\label{lem:PsiBk}Let $\mu$ be a partition of $n-k$ and let $\mu\cup k$ be the
partition of $n$ obtained by inserting the part $k$ in the appropriate place
of $\mu$. Then we have%
\[
\psi_{\mu}\in\Psi\left(  B^{k}\sigma\right)  \iff\psi_{\mu\cup k}\in
\Psi\left(  \sigma\right)
\]

\end{lemma}

\begin{proof}
Let $\alpha$ be the composition of $n$ obtained by inserting the part $k$ in
the end of $\mu$. It is a reordering of $\mu\cup k$ and thus $\psi_{{\alpha}}$
and $\psi_{\mu\cup k}$ belong to the same nilpotent orbit. The composition
with the natural projection $B^{k}\sigma\twoheadrightarrow\sigma$ defines an
isomorphism $Wh^{\prime}_{\psi_{\mu}}(B^{k}\sigma) \cong Wh^{\prime}%
_{\psi_{\alpha}}(\sigma)$. Thus
\[
\psi_{\mu}\in\Psi\left(  B^{k}\sigma\right)  \iff\psi_{{\alpha}} \in
\Psi\left(  \sigma\right)  \iff\psi_{\mu\cup k}\in\Psi\left(  \sigma\right)
\]

\end{proof}

If $\lambda$ is a partition and $k\leq\lambda_{1}$ then there is a unique $i$
such that $\lambda_{i}\geq k>\lambda_{i+1}$, and we define
\[
B^{k}(\lambda):=(\lambda_{1},\dots,\lambda_{i-1},\lambda_{i+2},\dots
,\lambda_{l})\cup(\lambda_{i}+\lambda_{i+1}-k),B^{k}({\mathcal{O}}_{\lambda
})={\mathcal{O}}_{B^{k}(\lambda)}.
\]
We extend this definition to unions of orbits, setting $B^{k}({\mathcal{O}%
}_{\lambda})=\emptyset$ if $k>\lambda_{1}.$

\begin{lem}
\label{lem:CombPart} Let $\lambda$ be a partition of $n$ and let $\mu$ be a
partition of $n-k$, then $\mu\cup k\leq\lambda$ if and only if $k\leq
\lambda_{1}$ and $\mu\leq B^{k}(\lambda)$.
\end{lem}

\begin{proof}
We use the notion of transposed partition $(\lambda^{t})_{i}=\max
\{j|\lambda_{j}\geq i\}$ and note that

\begin{enumerate}[(a)]
\item transposition is order-reversing.

\item $(\mu\cup k)^{t}$ is obtained from $\mu^{t}$ by adding 1 to each of the
first $k$ parts.

\item $(B^{k}(\lambda))^{t}$ is obtained from $\lambda^{t}$ by subtracting 1
from each of the first $k$ parts.
\end{enumerate}

The lemma follows.
\end{proof}

\begin{thm}
\label{thm:VarDer} If $\sigma\in\mathcal{HC}(G_{n})$ then $\mathrm{An}%
\mathcal{V}(B^{k}\sigma)=B^{k}\left(  \mathrm{An}\mathcal{V}(\sigma)\right)  $.
\end{thm}

\begin{proof}
Since $\psi_{\lambda}\in{\mathcal{O}}_{\lambda}$, by Theorem
\ref{thm:determine} and Lemma \ref{lem:PsiBk} we have
\[
{\mathcal{O}}_{\mu}\subset\mathrm{An}\mathcal{V}(B^{k}\sigma)\Leftrightarrow
\psi_{\mu}\in\Psi\left(  B^{k}\sigma\right)  \iff\psi_{\mu\cup k}\in
\Psi\left(  \sigma\right)  \iff{\mathcal{O}}_{\mu\cup k}\subset\mathrm{An}%
\mathcal{V}(\sigma)
\]
Thus it suffices to show that for any ${\mathcal{O}}$, ${\mathcal{O}}_{\mu\cup
k}\leq{\mathcal{O}}\Leftrightarrow{\mathcal{O}}_{\mu}\leq B^{k}({\mathcal{O}%
})$; this follows from Lemma \ref{lem:CombPart}.
\end{proof}

%In \cite{AGS} it is shown that $B^{k}(\pi)=0$ for all $k>depth(\pi)$. Theorem
%\ref{thm:GL} of the current paper implies $B^{depth(\pi)}(\pi) \neq0$.
%Moreover, it implies
%\begin{thm}
%\label{thm:DerDepth} Let $\pi\in\mathcal{M}(G_{n}({\mathbb{R}}))$. Suppose
%that $\mathrm{An}\mathcal{V}(\pi)=\overline{{\mathcal{O}}_{n_{1},...,n_{k}}}$
%with $n_{1}\geq... \geq n_{k}$. Then $depth(\pi) = n_{1}$ and $\mathrm{An}%
%\mathcal{V}(B^{n_{1}}(\pi))=\overline{{\mathcal{O}}_{n_{2},...,n_{k}}}$.
%\end{thm}
%For unitarizable $\pi$ this theorem is proven in \cite{GS}. The definition of
%derivative used in \cite{GS} is different, but equivalent (see \cite[section
%4]{AGS}).

In \cite{AGS} the \textit{depth} of $\sigma$ is defined to be the maximal rank
of a matrix $A\in\mathrm{An}\mathcal{V}(\sigma)\subset\mathfrak{gl}%
_{n}(\mathbb{C}),$ which is identified with $Mat_{n\times n}(\mathbb{C})$ via
the trace form. Note that if $\mathrm{An}\mathcal{V}(\sigma)=\overline
{{\mathcal{O}}_{\lambda}}$, then $\mathrm{depth}(\sigma)=\lambda_{1}$.

\begin{corollary}
If $\sigma\in\mathcal{HC}(G_{n})$ then $B^{k}(\sigma)=0$ if and only if
$k>\mathrm{depth}(\sigma)$.
\end{corollary}

By \cite[Corollary 4.2.2]{AGS}, Theorem \ref{thm:determine} implies

\begin{prop}
\label{prop:Monomial} Let $\chi_{i}$ be characters of $GL_{n_{i}}$ and
$n:=n_{1}+\cdots+n_{k}$. Let $\pi=\chi_{1}\times\cdots\times\chi_{k}%
\in\mathcal{M}(G_{n})$ be the corresponding induced representation. Then $\pi$
has a unique irreducible subquotient $\tau$ with
\[
\mathrm{An}\mathcal{V}(\tau)=\mathrm{An}\mathcal{V}(\pi)=\overline
{{\mathcal{O}}_{(n_{1},\dots,n_{k})}}.
\]
Moreover, $\tau$ occurs in $\pi$ with multiplicity one.
\end{prop}

\begin{proof}
Without loss of generality, we can suppose $n_{1}\geq\cdots\geq n_{k}$ and
write ${\lambda}:=(n_{1},\dots,n_{k})$. Then it is known that $\mathrm{An}%
\mathcal{V}(\pi)=\overline{{\mathcal{O}}_{{\lambda}}}$ and thus $\mathrm{An}%
\mathcal{V}(\tau)\subset\overline{{\mathcal{O}}_{{\lambda}}}$ for any
subquotient $\tau$ of $\pi$. Now by Theorem \ref{thm:determine},
$\mathrm{An}\mathcal{V}(\tau)\supset\overline{{\mathcal{O}}_{{\lambda}}}$ iff
$\psi_{\lambda}\in\Psi\left(  \tau\right)  ,$ and by \cite[Corollary
4.2.2]{AGS}, $\pi$ has a unique such constituent.
\end{proof}

\section{The case of complex classical groups}

\label{sec:Orb} \setcounter{lemma}{0}

%\subsection{Proof of Theorem \ref{thm:ClassGroup}}

%Let $G$ be a quasi-split group,
In this section we prove Theorem \ref{thm:ClassGroup}.
For convenience we fix an invariant form $\left\langle x,y\right\rangle $ on
$\mathfrak{g}_{0}$ and we identify $i\mathfrak{g}_{0}^{\ast}$ with
$\mathfrak{g}_{0}$ as follows:%
\[
\psi_{x}\left(  y\right)  =i\left\langle x,y\right\rangle \text{.}%
\]

Let $H,B,\Pi_{0}$ etc. be as before, then $\Psi_{0}\subset$ $\mathfrak{g}_{0}$
can be identified with the direct sum of negative simple roots spaces%
\[
\Psi_{0}=\bigoplus\nolimits_{\beta\in\Pi_{0}}\mathfrak{g}_{0}^{-\beta}%
\]
By Proposition \ref{prop:nondeg} this identifies non-degenerate elements of
$\Psi_{0}$ with the \emph{principal nilpotent elements} in $\Psi_{0}$, namely
those for which each of the projections $p_{\beta}:\Psi_{0}\rightarrow
\mathfrak{g}_{0}^{-\beta}$ is non-zero. Lemma \ref{lem:non-deg-par} gives the
following result.

\begin{lemma}
\label{lm:PL}If $e\in\Psi_{0}$ then there is a standard Levi subalgebra
$\mathfrak{l}_{0}$ such that $e$ is a principal nilpotent element in
$\mathfrak{l}_{0}.$
\end{lemma}

We say that a nilpotent orbit $\mathcal{O}\subset\mathfrak{g}_{0}$ is a
PL-orbit if $\mathcal{O}\cap\mathfrak{l}_{0}$ is a principal nilpotent orbit
in some Levi subalgebra $\mathfrak{l}_{0}$. Let $PL\left(  G\right)  $ denote
the set of PL-orbits, and for an arbitrary nilpotent orbit $\mathcal{O}$ define%
\[
PL\left(  \mathcal{O}\right)  =\left\{  \mathcal{O}^{\prime}\leq\mathcal{O\mid
O}^{\prime}\in PL\left(  G\right)  \right\}  .
\]

\begin{lemma}
For each nilpotent orbit $\mathcal{O}$, the sets $\overline{\mathcal{O}}%
\cap\Psi_{0}$ and $PL\left(  \mathcal{O}\right)  $ determine each other uniquely.
\end{lemma}

\begin{proof}
Let $X$ denote the union of the orbits in $PL\left(  \mathcal{O}\right)  $ and
let $Y=\overline{\mathcal{O}}\cap\Psi_{0}$. Then by the previous lemma we get
$X=G\cdot Y$ and $Y=X\cap\Psi_{0}$.
\end{proof}

Therefore Theorem \ref{thm:ClassGroup} reduces to the following statement.

\begin{theorem}
\label{thm:PL} For a complex classical group, every nilpotent orbit
$\mathcal{O}$ is determined by $PL\left(  \mathcal{O}\right)  $.
\end{theorem}

We will prove this in \S \ref{subsec:mp} after describing the sets $PL\left(
G\right)  $ for classical groups. First we recall the classification of
nilpotent orbits in classical complex Lie algebras.

\subsection{Nilpotent orbits for complex classical groups}

If $G=GL\left(  d,\mathbb{R}\right)  $ or if $G$ is a complex classical group
as in (\ref{=ccg}), then the real nilpotent orbits of $G$ are naturally
indexed by partitions, as in \cite{CM}. A \emph{partition} $\lambda$ of $d$ of
length $l$ is a weakly decreasing integer sequence $\lambda_{1}\geq\lambda
_{2}\geq\cdots\geq\lambda_{l}>0$ such that $\sum_{j}\lambda_{j}=d$. The
$\lambda_{i}$ are called the parts of $\lambda$, and the number of parts of
size $p$ is called the multiplicity $m_{p}\left(  \lambda\right)  $ of $p$. We
write $\mathcal{P}\left(  d\right)  $ for the set of all partitions of $d$ and
$\mathcal{P}_{1}\left(  d\right)  $ (resp. $\mathcal{P}_{-1}\left(  d\right)
$ ) for the subset such that $m_{p}\left(  \lambda\right)  $ is even for all
even (resp. odd) $p$. We set $\lambda_{j}=0$ if $j$ exceeds the length of
$\lambda$, and we define a partial order on partitions as follows:%
\[
\lambda\leq\mu\text{ iff }\lambda_{1}+\cdots+\lambda_{k}\leq\mu_{1}+\cdots
+\mu_{k}\text{ for all }k\text{.}%
\]

\begin{theorem}
\label{thm:orb}There is an order-preserving bijection between nilpotent
$G$-orbits and the set $\mathcal{P}\left(  G\right)  $ below:%
\[%
\begin{tabular}
[c]{|c|c|c|c|}\hline
$G$ & $GL\left(  d,\mathbb{R}\right)  ,GL\left(  d,\mathbb{C}\right)
,SL\left(  d,\mathbb{C}\right)  $ & $O\left(  d,\mathbb{C}\right)  $ &
$Sp\left(  d,\mathbb{C}\right)  $\\\hline
$\mathcal{P}\left(  G\right)  $ & $\mathcal{P}\left(  d\right)  $ &
$\mathcal{P}_{1}\left(  d\right)  $ & $\mathcal{P}_{-1}\left(  d\right)
$\\\hline
\end{tabular}
\ \
\]

\end{theorem}

The case of $SO\left(  d,\mathbb{C}\right)  $ is slightly different. We say
that $\lambda\in\mathcal{P}_{1}\left(  d\right)  $ is \textquotedblleft very
even\textquotedblright\ if $\lambda$ has only even parts. Note that each even
part must occur with even multiplicity, forcing $d$ to be a multiple of $4$.

\begin{theorem}
The nilpotent orbits of $SO\left(  d,\mathbb{C}\right)  $ are the same as
$O\left(  d,\mathbb{C}\right)  $ except that the very even\ orbits
$\mathcal{O}_{\lambda}$ split into two orbits for $SO\left(  d,\mathbb{C}%
\right)  $, denoted $\mathcal{O}_{\lambda}^{I}$ and $\mathcal{O}_{\lambda
}^{II}$.
\end{theorem}

For proofs we refer the reader to \cite{CM}, especially Chapters 5 and 6.

\subsection{Principal nilpotents in Levi subgroups}

\label{subsec:PrinClass}

In this section we assume that $G=GL\left(  d,\mathbb{C}\right)  ,O\left(
d,\mathbb{C}\right)  $ or $Sp\left(  d,\mathbb{C}\right)  $, and write
$GL\left(  d\right)  $ etc. for simplicity. Nilpotent orbits for $G$ are
parameterized by partitions of $d$ as in Theorem \ref{thm:orb}; we will regard
$PL\left(  G\right)  $ as a set of partitions and write $PL\left(
\lambda\right)  $ instead $PL\left(  \mathcal{O}_{\lambda}\right)  .$

\begin{lemma}
Let $\lambda_{\max}$ be the partition corresponding to a principal nilpotent
orbit; then%
\[
\lambda_{\max}=\left\{
\begin{tabular}
[c]{ll}%
$\left(  d-1,1,0,\ldots\right)  $ & if $G=O\left(  d\right)  $ with $d$ even\\
$\left(  d,0,0,\ldots\right)  $ & otherwise
\end{tabular}
\ \right.
\]

\end{lemma}

\begin{proof}
The principal nilpotent orbit is maximal with respect to the closure order.
The result follows from Theorem \ref{thm:orb} and the easy verification that
$\lambda_{\max}$ is the maximal element in $\mathcal{P}\left(  G\right)  $.
\end{proof}

For a partition $\lambda$ write $OM\left(  \lambda\right)  =\left\{  p>1\mid
m_{p}\left(  \lambda\right)  \text{ is odd}\right\}  $ and define%
\[
\mathcal{X}\left(  G\right)  =\left\{
\begin{tabular}
[c]{cl}%
$\mathcal{P}\left(  G\right)  =\mathcal{P}\left(  d\right)  $ & if
$G=GL\left(  d\right)  $\\
$\left\{  \lambda\in\mathcal{P}\left(  G\right)  :\left\vert OM\left(
\lambda\right)  \right\vert \leq1\right\}  $ & otherwise
\end{tabular}
\ \ \right.
\]

\begin{proposition}
If $G=GL\left(  d\right)  ,O\left(  d\right)  $ or $Sp\left(  d\right)  $ then
$PL\left(  G\right)  =\mathcal{X}\left(  G\right)  $.
\end{proposition}

\begin{proof}
For $G=GL\left(  d\right)  $ the proposition asserts that every orbit is
principal in some Levi subgroup, which follows from the Jordan canonical form.

The Levi subgroups of $O\left(  d\right)  $ and $Sp\left(  d\right)  $ are
given as follows: up to conjugacy there is one for each partition $\kappa$
with $\kappa_{1}\geq\cdots\geq\kappa_{r}$ \ such that $d^{\prime}=d-2\left(
\mathfrak{\kappa}_{1}+\cdots+\mathfrak{\kappa}_{r}\right)  \geq0$. Explicitly%
\[
L_{\kappa}=\left\{
\begin{tabular}
[c]{ll}%
$O\left(  d^{\prime}\right)  \times GL\left(  \kappa_{1}\right)  \times
\cdots\times GL\left(  \kappa_{r}\right)  $ & if $G=O\left(  d\right)  $\\
$Sp\left(  d^{\prime}\right)  \times GL\left(  \kappa_{1}\right)  \times
\cdots\times GL\left(  \kappa_{r}\right)  $ & if $G=Sp\left(  d\right)  $%
\end{tabular}
\ \ \right.  \text{ }%
\]

The principal nilpotent orbit in $L_{\kappa}$ can be determined by the
previous lemma. In the partition $\lambda_{\kappa}$ for corresponding
nilpotent orbit in ${\mathfrak{g}}_{0}^{\ast}$, each $GL\left(
{\kappa}_{i}\right)  $ factor contributes $\emph{two}$ parts of size
${\kappa}_{i}$. Thus up to decreasing reordering of the parts, we
have%
\[
\lambda_{\kappa}=\left\{
\begin{tabular}
[c]{ll}%
$\left(  d^{\prime}-1,1,{\kappa}_{1},{\kappa}_{1}%
,\ldots,{\kappa}_{r},{\kappa}_{r},0,0,\cdots\right)  $ & if
$G=O\left(  d\right)  $ with $d$ even\\
$\left(  d^{\prime},{\kappa}_{1},{\kappa}_{1},\ldots
,{\kappa}_{r},{\kappa}_{r},0,0,\cdots\right)  $ & otherwise
\end{tabular}
\ \ \right.
\]
By definition, parts with even multiplicity do not contribute to $OM\left(
\lambda\right)  $, thus
\[
OM\left(  \lambda_{\kappa}\right)  =\left\{
\begin{tabular}
[c]{ll}%
$OM\left(  \left(  d^{\prime}-1,1\right)  \right)  $ & if $G=O\left(
d\right)  $ with $d$ even\\
$OM\left(  \left(  d^{\prime}\right)  \right)  $ & otherwise
\end{tabular}
\ \ \right.
\]
Moreover since the part $1$ does not contribute to $OM\left(  \lambda\right)
$,\ we get $\left\vert OM\left(  \lambda_{\kappa}\right)  \right\vert \leq1$.
Thus $PL\left(  G\right)  \subseteq\mathcal{X}\left(  G\right)  $.

Conversely suppose $\lambda\in\mathcal{P}\left(  G\right)  $ satisfies
$\left\vert OM\left(  \lambda\right)  \right\vert \leq1$. Then $\lambda$ has
$0,1,$ or $2$ parts with odd multiplicity, and in the last case the part $1$
has odd multiplicity. Thus the last case can only occur if $G=O\left(
d\right)  $, and since there are exactly two odd parts with odd multiplicity,
$d$ must be even. It follows now that $\lambda$ is of the form $\lambda
_{\kappa}$ for some $\kappa$. Thus $\mathcal{X}\left(  G\right)  \subseteq
PL\left(  G\right)  $.
\end{proof}

We prove Theorem \ref{thm:PL} in the next subsection, using the following lemma.

\begin{lemma}
\label{lem:RecPart} For each $\lambda\in\mathcal{P}\left(  G\right)  $ and
each $k$ there is a partition $\mu=\mu\left(  \lambda,k\right)  \in PL\left(
\lambda\right)  $ such that
\[
\mu_{1}+\cdots+\mu_{k}=\lambda_{1}+\cdots+\lambda_{k}%
\]

\end{lemma}

\begin{proof}
Let $j$ be the largest index such that $\lambda_{j}=\lambda_{k}$. If $\left(
\lambda_{1},\ldots,\lambda_{j}\right)  $ contains two or more parts $p,q$ with
odd multiplicity, then necessarily $p,q$ have the same parity and so
$r=\left(  p+q\right)  /2$ is an integer. If $\nu$ is obtained from $\lambda$
by replacing a pair $\left(  p,q\right)  $ by $\left(  r,r\right)  $, then we
have $\nu\leq\lambda$ and $\nu_{1}+\cdots+\nu_{k}=\lambda_{1}+\cdots
+\lambda_{k}$. Iterating this we may assume that $\left(  \lambda_{1}%
,\ldots,\lambda_{j}\right)  $ contains at most one part with odd multiplicity.

Now let $\mu\in\mathcal{P}\left(  G\right)  $ be obtained from $\lambda$ by
replacing the parts $\lambda_{j+1},\lambda_{j+2}\dots$ by a string of $1$'s of
length $\left(  \lambda_{j+1}+\lambda_{j+2}+\cdots\right)  $ Then $\left\vert
OM\left(  \mu\right)  \right\vert \leq1$ and hence $\mu$ satisfies the
condition of the Lemma.
\end{proof}

\subsection{Proof of Theorem \ref{thm:PL}}

\label{subsec:mp}

We now prove Theorem \ref{thm:PL} for all classical groups.

\begin{proof}
[Proof of Theorem \ref{thm:PL}]First suppose $G=GL\left(  d\right)  ,SL\left(
d\right)  ,O\left(  d\right)  $ or $Sp\left(  d\right)  $. We need to show
that each $\lambda\in\mathcal{P}\left(  G\right)  $ is determined by the set
$PL\left(  \lambda\right)  $. This is obvious for $G=GL\left(  d\right)  $,
$SL\left(  d\right)  $ and therefore we may assume that $G=O\left(  d\right)
$ or $Sp\left(  d\right)  $. By definition of the partial order, for each $k$
we have%
\[
\mu_{1}+\cdots+\mu_{k}\leq\lambda_{1}+\cdots+\lambda_{k}\text{ for all }\mu\in
PL\left(  \lambda\right)  .
\]

Moreover by Lemma \ref{lem:RecPart} equality holds for some $\mu$. Therefore
for each $k$ we can recover the sum $\lambda_{1}+\cdots+\lambda_{k}$ as the
maximum of $\mu_{1}+\cdots+\mu_{k}$ for $\mu\in PL\left(  \lambda\right)  $,
and hence we can determine $\lambda$ as well.

Finally we consider $G=SO\left(  d\right)  $. If $\mathcal{O=O}_{\lambda}$
where $\lambda$ is not very even, then $\mathcal{O}$ is a single $O\left(
d\right)  $ orbit and so the result follows by the $O\left(  d\right)  $
argument. If $\mathcal{O}=\mathcal{O}_{\lambda}^{I}$ or $\mathcal{O}_{\lambda
}^{II}$ for some very even $\lambda$, then $\mathcal{O}\cap\Psi_{0}$ is
nonempty, thus $\mathcal{O}$ can be recovered from $\overline{\mathcal{O}}%
\cap\Psi_{0}$ in this case as well.
\end{proof}

The theorem does not extend to unions of orbits.

\begin{example}
The table below lists some examples of partition triples $\left[  \lambda,\mu
,\nu\right]  $ such that $PL\left(  \lambda\right)  =PL\left(  \mu\right)
\cup PL\left(  \nu\right)  $. All orbits are special in the sense of
Lusztig-Spaltenstein (see \cite[Section 6.3]{CM}).
\begin{equation}
\begin{tabular}
[c]{|c|c|c|c|}\hline
$G$ & $\lambda$ & $\mu$ & $\nu$\\\hline
$O\left(  11\right)  $ & $(7,3,1)$ & $(5,5,1)$ & $\left(  7,2,2\right)
$\\\hline
$Sp\left(  10\right)  $ & $\left(  6,4\right)  $ & $\left(  5,5\right)  $ &
$\left(  6,2,2\right)  $\\\hline
$O\left(  8\right)  $ & $\left(  5,3\right)  $ & $\left(  4,4\right)  $ &
$\left(  5,1,1,1\right)  $\\\hline
\end{tabular}
\ \ \label{CounterClass}%
\end{equation}

\end{example}

\begin{remark}
\label{rem:rank} If $G$ is a classical group, we can regard elements of
${\mathfrak{g}_{0}}$ and ${\mathfrak{g}_{0}^{\ast}}$ as matrices. For a matrix
$X\in{\mathcal{O}}_{\lambda}$ its rank and order of nilpotence are given by
$n-length(\lambda)$ and $\lambda_{1}$ respectively; we refer to these as the
rank and depth of ${\mathcal{O}}_{\lambda}$. If $\mathcal{V}$ is a union of
orbits we define $rank(\mathcal{V})$ and $depth(\mathcal{V})$ by taking
maxima, and the arguments above show that these are uniquely determined by
$\mathcal{V}\cap\Psi_{0}$.

For $\pi\in{\mathcal{M}}\left(  G\right)  $ we define $rank(\pi
)=rank(\operatorname{WF}\pi)$ and $depth(\pi)=depth(\operatorname{WF}\pi)$. By
\cite{He} $rank(\pi)$ coincides with the Howe rank of $\pi$, and for $GL(n)$,
$depth\left(  \pi\right)  $ coincides with the notion of depth in section
\ref{sec:GL}. It would be interesting to give a representation-theoretic
characterization of depth for other classical groups. This remark shows that
for all $\pi\in{\mathcal{M}}\left(  G\right)  $ the rank and the depth are
determined by $\Psi(\pi)$.
\end{remark}

\subsection{Exceptional groups}

\label{subsec:ExcGroups}

Theorem \ref{thm:ClassGroup} is false for every exceptional complex group and
we now describe all counterexamples via the Bala-Carter classification
\cite[\S 13.4]{Car}. Let us say for simplicity that two nilpotent
$G$-orbits ${\mathcal{O}}$ and ${\mathcal{O}}^{\prime}$ are \emph{related} if
$PL\left(  {\mathcal{O}}\right)  =PL\left(  {\mathcal{O}}^{\prime}\right)  $.

\begin{prop}
\label{prop:ExcGroup} The following is a complete list of related orbits, with
special orbits underlined.%
\[%
\begin{tabular}
[c]{|c|c|c|c|}\hline
$G$ & \emph{Related orbits} & $G$ & \emph{Related orbits}\\\hline
$E_{6}$ & $\underline{E_{6}(a_{1})}$ and $\underline{D_{5}}$ & $G_{2}$ &
$\underline{G_{2}(a_{1})}$ and $\widetilde{A_{1}}$\\[1ex]\hline
$E_{6}$ & $\underline{D_{4}(a_{1})}$ and $A_{3}+A_{1}$ & $E_{8}$ &
$\underline{E_{8}(a_{1})}$, $\underline{E_{8}(a_{2})}$ and $\underline{E_{8}%
(a_{3})}$\\[1ex]\hline

$E_{7}$ & $\underline{E_{7}(a_{1})}$ and $\underline{E_{7}(a_{2})}$ & $E_{8}$
& $\underline{E_{8}(a_{4})}$, $\underline{E_{8}(b_{4})}$ and $\underline{E_{8}%
(a_{5})}$\\[1ex]\hline
$E_{7}$ & $\underline{E_{7}(a_{3})}$ and $D_{6}$ & $E_{8}$ & $\underline{E_{7}%
(a_{1})}$, $\underline{E_{8}(b_{5})}$ and $E_{7}(a_{2})$\\[1ex]\hline
$E_{7}$ & $\underline{E_{6}(a_{1})}$ and $\underline{E_{7}(a_{4})}$ & $E_{8}$
& $\underline{E_{8}(a_{6})}$ and $\underline{D_{7}(a_{1})}$\\[1ex]\hline
$F_{4}$ & $\underline{F_{4}(a_{1})}$ and $\underline{F_{4}(a_{2})}$ & $E_{8}$
& $\underline{E_{6}(a_{1})}$ and $\underline{E_{7}(a_{4})}$\\[1ex]\hline
$F_{4}$ & $\underline{F_{4}(a_{3})}$ and $C_{3}(a_{1})$ & $E_{8}$ &
$\underline{E_{8}(a_{7})}$, $E_{7}(a_{5})$, $E_{6}(a_{3})+A_{1}$, $D_{6}%
(a_{2})$\\[1ex]\hline
\end{tabular}
\]
\end{prop}

\begin{proof}
In Bala-Carter notation, the PL-orbits are labeled by the corresponding Levi
subalgebra $\mathfrak{l}$, while the other orbits have labels of the form
$\mathfrak{l}\left(  \ast\right)  $. Thus for any orbit ${\mathcal{O}}$ we can
easily compute $PL\left(  {\mathcal{O}}\right)  $ by looking at lower orbits
whose Bala-Carter labels have no parentheses. With this in mind, the table above
follows from the Bala-Carter classification tables \cite[\S 13.4]{Car}.
\end{proof}


\begin{thebibliography}{999999}                                                                                           %


\bibitem[AGS]{AGS}A. Aizenbud, D. Gourevitch, S.Sahi: \textit{Derivatives for
representations of }$GL(n,\mathbb{R})$\textit{ and }$GL(n,\mathbb{C})$, arXiv: 1109.4374[math.RT], to appear in the Israel Journal of Mathematics.

\bibitem[AM69]{AM}M. Atiyah, I. Macdonald: \textit{Introduction to Commutative
algebra}. Addison-Wesley, Reading, Massachusetts (1969).

%\bibitem[AS]{ASGel}A. Aizenbud, E. Sayag: \textit{Invariant distributions on
%non-distinguished nilpotent orbits with application to the Gelfand property of
%(GL(2n,R),Sp(2n,R))}, arXiv:0810.1853 [math.RT].
%\bibitem[AOS]{AOSUnDisj}A. Aizenbud, O. Offen, E. Sayag: \textit{Uniqueness
%and disjointness of Klyachko models in the Archimedean case}, in preparation.
%\bibitem[BB89]{BB2}W. Borho, J.-L. Brylinski \textit{Differential operators on
%homogeneous spaces. II. Relative enveloping algebras.} Bull. Soc. Math. France
%\textbf{117} (1989), no. 2, 167--210.


\bibitem[BB81]{BB}A. Beilinson, J. Bernstein: \textit{Localisations de
}$\mathfrak{g}$\textit{-modules}, C. R. Acad. Sci. Paris \textbf{292}, pp 15 -
18 (1981).

\bibitem[BB85]{BorBry}W. Borho, J.-L. Brylinski: \textit{Differential
operators on homogeneous spaces III}, Invent. Math. 80 (1985), 1-68.

\bibitem[Ber72]{Ber}J. Bernstein: \textit{Modules over the ring of
differential operators; the study of fundamental solutions of equations with
constant coefficients,} Functional Analysis and its Applications \textbf{5},
No.2, 1-16 (1972).

\bibitem[BK]{BerKr}J. Bernstein, B. Kroetz: \textit{Smooth Frechet
Globalizations of Harish-Chandra Modules,} arXiv:0812.1684, to appear in
Israel Journal of Mathematics.

\bibitem[BS98]{BS}D. Barbasch, M. Sepanski, \textit{Closure ordering and the
Kostant-Sekiguchi correspondence}, Proc Amer. Math. Soc. 126 (1998), 311-317.

\bibitem[BV85]{BV-Unip}D. Barbasch, D. Vogan, \textit{Unipotent
representations of complex semisimple groups. }Ann. of Math. \textbf{121}, no.
1, pp 41--110 (1985).

\bibitem[BZ77]{BZ-Induced}I. Bernstein, A. Zelevinsky: \textit{Induced
representations of reductive p-adic groups. I,} Ann. Sci. Ec. Norm. Super.,
4$^{\text{e}}$ serie \textbf{10}, pp. 441-472 (1977).

\bibitem[Cas80]{Cas}W. Casselman: \textit{Jacquet modules for real reductive
groups. Proceedings of the International Congress of Mathematicians (Helsinki,
1978)}, pp. 557--563, Acad. Sci. Fennica, Helsinki (1980).

\bibitem[Cas89]{CasGlob}W. Casselman: \textit{Canonical extensions of
Harish-Chandra modules to representations of G,} Can. J. Math., Vol. XLI, No.
3, pp. 385-438 (1989).

\bibitem[Car85]{Car}R.W. Carter: \textit{Finite groups of Lie type. Conjugacy
classes and complex characters.} Pure and Applied Mathematics (New York). A
Wiley-Interscience Publication. John Wiley \& Sons, Inc., New York, (1985).
xii+544 pp. ISBN: 0-471-90554-2

\bibitem[CHM00]{CHM}W. Casselman; H. Hecht; D. Mili\v{c}i\'{c}: \textit{Bruhat
filtrations and Whittaker vectors for real groups}. The mathematical legacy of
Harish-Chandra (Baltimore, MD, 1998), 151-190, Proc. Sympos. Pure Math.,
\textbf{68}, Amer. Math. Soc., Providence, RI, (2000)

\bibitem[CoMG93]{CM}D. Collingwood, W. McGovern: \textit{Nilpotent orbits in
semisimple Lie algebras.} Van Nostrand Reinhold Mathematics Series. Van
Nostrand Reinhold Co., New York (1993). xiv+186 pp.

\bibitem[ENV04]{ENV}M. Emerton, D. Nadler, K. Vilonen: \textit{A geometric
Jacquet functor}, Duke Math. J. \textbf{125}, no. 2, 267--278 (2004).

\bibitem[Gin86]{Gin}V. Ginsburg: \textit{Characteristic varieties and
vanishing cycles}, Invent. Math. \textbf{84} , no. 2, pp. 327-402 (1986).

\bibitem[GK75]{GK}I.~M. Gelfand, D.~A. Kajdan, \textit{Representations of the
group }$GL(n,K)$\textit{ where }$K$\textit{ is a local field,} Lie groups and
their representations (Proc. Summer School, Bolyai J\'{a}nos Math. Soc.,
Budapest, 1971), Halsted, New York, 1975, pp.~95--118. \MR{MR0404534 (53 \#8334)}

\bibitem[GS13]{GS} D. Gourevitch, S. Sahi:
\textit{Associated varieties, derivatives, Whittaker functionals, and
rank for unitary representations of $GL(n)$}, Selecta Mathematica (New Series) \textbf{19}, 141-172 (2013).

\bibitem[GW80]{GW}R. Goodman, N. Wallach: \textit{Whittaker vectors and
conical vectors}, J. Funct. Anal. \textbf{39}, no. 2, pp. 199-279 (1980).

\bibitem[Har77]{Hart}R. Hartshorne: \textit{Algebraic geometry}. Graduate
Texts in Mathematics, No. \textbf{52}. Springer-Verlag, New York-Heidelberg
(1977) xvi+496 pp.

\bibitem[HCh78]{Hch}Harish-Chandra: \textit{Admissible distributions on
reductive p-adic groups}, Lie Theories and their applications, Queen's Papers
in Pure and Applied Mathematics, Queen's University, Kingston, Ontario, 1978,
p. 281-347.

\bibitem[He08]{He}H. He: \textit{Associated varieties and Howe's N-spectrum},
Pacific Journal of Mathematics, \textbf{237} No. 1 (2008).

\bibitem[How74]{HowGL}R. Howe: \textit{The Fourier transform and germs of
characters (case of $GL_{n}$ over a p-adic field)} . Math. Ann 208, 305-322 (1974).

\bibitem[How81]{HowWF}R. Howe: \textit{Wave front sets of representations of
Lie groups,}in: Automorphic Forms, Representation Theory and Arithmetic,
Bombay (1981).
%\bibitem[How82]{How}Howe, R. On a notion of rank for unitary representations
%of the classical groups. Harmonic analysis and group representations 223-331 (1982).


\bibitem[Jos81]{JosLect}A. Joseph: \textit{Application de la theorie des
anneaux aux algebres enveloppantes}. Lecture Notes, Paris (1981). Available at \url{http://www.wisdom.weizmann.ac.il/~dimagur/AntonyJosephLectNotes.zip}.

\bibitem[Jos85]{JosVar}A. Joseph: \textit{On the associated variety of a
primitive ideal}, Journal of Algebra \textbf{93} , no. 2, 509--523 (1985).

\bibitem[Kos78]{Kos}B. Kostant: \textit{On Whittaker vectors and
representation theory.}, Invent. Math. \textbf{48}, 101-184 (1978).

\bibitem[KR71]{KR}B. Kostant, S. Rallis: \textit{ Orbits and representations
associated with symmetric spaces.} Amer. J. Math. \textbf{93} (1971), 753--809.

\bibitem[KV79a]{KVShi}M. Kashiwara, M.Vergne: \textit{Functions on the Shilov
boundary of the generalized half plane,} in: Non-Commutative Harmonic
Analysis, Lecture Notes in Mathematics \textbf{728}, Springer-Verlag,
Berlin-Heidelberg-New York (1979).

\bibitem[KV79b]{KVWF}M. Kashiwara, M.Vergne: \textit{$K$-types and the
singluar spectrum,} in: Non-Commutative Harmonic Analysis, Lecture Notes in
Mathematics \textbf{728}, Springer-Verlag, Berlin-Heidelberg-New York (1979).

\bibitem[Mat87]{Mat}H. Matumoto: \textit{Whittaker vectors and associated
varieties}, Invent. math. \textbf{89}, 219-224 (1987)

\bibitem[Mat88]{MatActa}H. Matumoto: \textit{Whittaker vectors and the
Goodman-Wallach operators}, Acta. Math. \textbf{161} (1988), no. 3-4, 183--241.

\bibitem[Mat90]{MatDuke}H. Matumoto: \textit{Whittaker modules associated with
highest weight modules}, Duke Math. J. \textbf{60} (1990), no. 1, 59--113.

\bibitem[Mat92]{MatComp}H. Matumoto: \textit{$C^{-\infty}$-Whittaker vectors
corresponding to a principal nilpotent orbit of a real reductive linear Lie
group, and wave front sets,} Compositio Math. \textbf{82} (1992), no. 2, 189-244.

\bibitem[MT07]{MatTra}H. Matumoto, P. Trapa: \textit{Derived functor modules
arising as large irreducible constituents of degenerate principal series,}
Compositio Math. \textbf{143} (2007), 222-256.

\bibitem[McC67]{McC}J. McConnell: \textit{The intersection theorem for a class
of non-commutative rings,} Proc. London Math. Soc. \textbf{17} 487--498 (1967).

\bibitem[MW87]{MW}C. Moeglin, J.L. Waldspurger: \textit{Modeles de Whittaker
degeneres pour des groupes p-adiques,} Math. Z. \textbf{196} (1987), no. 3, 427-452.

\bibitem[Ohta81]{Ohta}T. Ohta: \textit{The closures of nilpotent orbits in the
classical symmetric pairs and their singularities,} Tohoku Math. J.
\textbf{43} (1991), 161-211.

\bibitem[Rod75]{Rod}F. Rodier: \textit{Mod\`{e}le de Whittaker et
caract\`{e}res de repr\'{e}sentations} (French). Non-commutative harmonic
analysis (Actes Colloq., Marseille-Luminy, 1974), pp. 151-171. Lecture Notes
in Math., Vol. 466, Springer, Berlin, (1975).

\bibitem[Ros95a]{Ros1}W. Rossmann: \textit{Picard-Lefschetz theory for the
coadjoint quotient of a semisimple Lie algebra}, Inventiones Mathematicae
\textbf{121}, 531-578 (1995).

\bibitem[Ros95b]{Ros2}W. Rossmann: \textit{Picard-Lefschetz theory and
characters of a semisimple Lie group}, Inventiones Mathematicae \textbf{121},
579-611 (1995).

\bibitem[Sek87]{Sek}J. Sekiguchi: \textit{Remarks on real nilpotent orbits of
a symmetric pair}. Journal of the Mathematical Society of Japan \textbf{39},
127-138 (1987).

\bibitem[Sha74]{Sha}J. Shalika: \textit{The multiplicity one theorem for
}$GL(n)$, Ann. of Math. \textbf{100} (1974), 171-193.

\bibitem[SV00]{SV}W. Schmid, K. Vilonen: \textit{Characteristic cycles and
wave front cycles of representations of reductive Lie groups}. Annals of
Mathematics, \textbf{151} (2000), 1071-1118.
%\bibitem[Vog78]{Vog}D. A. Vogan \textit{Gelfand-Kiriliov Dimension for
%Harish-Chandra Modules}, Inventiones math \textbf{48}, 75-98 (1978)


\bibitem[Vog91]{Vog-Unip}D. A. Vogan: \textit{Associated varieties and
unipotent representations}. Harmonic analysis on reductive groups (Brunswick,
ME, 1989), pp 315-388, Progr. Math., \textbf{101}, Birkhauser Boston, Boston,
MA (1991).

% \bibitem[Vog81]{Vog-book}D. A. Vogan: \textit{Representations of real
% reductive groups, }Birkhauser, Boston-Basel-Stuttgart, (1981)

\bibitem[Wall88]{Wal1}N. Wallach: \textit{Real Reductive groups I}, Pure and
Applied Math. \textbf{132}, Academic Press, Boston, MA (1988).

\bibitem[Wall92]{Wal2}N. Wallach: \textit{Real Reductive groups II}, Pure and
Applied Math. \textbf{132}, Academic Press, Boston, MA (1992).

\bibitem[Zel80]{Zl}A. Zelevinsky: \textit{Induced representations of reductive
p-adic groups. II. On irreducible representations of }$GL(n)$, Ann. Sci. Ec.
Norm. Super., 4$^{\text{e}}$ serie \textbf{13}, 165-210 (1980).
\end{thebibliography}
\end{document}